\documentclass[final,a4paper,reqno,11pt]{amsart}
\pdfoutput=1

\usepackage{soul}
\usepackage[final]{graphicx}
\usepackage[final]{hyperref}
\usepackage{amssymb, amstext, amsmath, amsthm, amsfonts, mathrsfs, lmodern, ifdraft, stmaryrd, enumitem}
\usepackage{tikz}
\usepackage{tikz-cd}
\usetikzlibrary{decorations.pathmorphing}
\usepackage[noadjust]{cite}
\usepackage[utf8]{inputenc}
\usepackage[english]{babel}
\usepackage[noabbrev]{cleveref}
\usepackage{wasysym}
\usepackage{float}
\usepackage[activate={true,nocompatibility},final,tracking=true,kerning=true,spacing=true,factor=1100,stretch=10,shrink=10]{microtype}
\microtypecontext{spacing=nonfrench}
\usepackage[figurename=Diagram]{caption}
\usepackage{setspace}
\usepackage{xspace, mathtools}

\newtheorem{theorem}{Theorem}[section]
\newtheorem{corollary}[theorem]{Corollary}
\newtheorem{lemma}[theorem]{Lemma}
\newtheorem{proposition}[theorem]{Proposition}

\newtheorem*{theorem*}{Theorem}
\newtheorem*{corollary*}{Corollary}
\newtheorem{introtheorem}{Theorem}

\newtheorem{introprop}[introtheorem]{Proposition}

\Crefname{introcorollary}{corollary}{corollary}

\theoremstyle{definition}
\newtheorem{definition}[theorem]{Definition}
\newtheorem{notation}[theorem]{Notation}
\newtheorem{setup}[theorem]{Setup}

\newtheorem*{conjecture*}{Conjecture}

\theoremstyle{remark}
\newtheorem{remark}[theorem]{Remark}
\newtheorem{calculation}[theorem]{Calculation}
\newtheorem*{acknowledgment}{Acknowledgment}

\newcommand{\fs}{\mathfrak{s}}
\newcommand{\fr}{\mathfrak{r}}
\newcommand{\ft}{\mathfrak{t}}

\newcommand{\cA}{\mathcal{A}}
\newcommand{\cB}{\mathcal{B}}
\newcommand{\cC}{\mathcal{C}}
\newcommand{\cD}{\mathcal{D}}

\newcommand{\cX}{\mathcal{X}}

\newcommand{\C}{\mathbf{C}}

\newcommand{\sF}{\mathscr{F}}

\newcommand{\sD}{\mathsf{D}}

\newcommand{\sJ}{\mathscr{I}}

\newcommand{\NN}{\mathbb{N}}
\newcommand{\EE}{\mathbb{E}}
\newcommand{\FF}{\mathbb{F}}
\newcommand{\GG}{\mathbb{G}}

\newcommand{\ZZ}{\mathbb{Z}}

\newcommand{\Ab}{\mathsf{Ab}}

\renewcommand{\mod}{\mathsf{mod}\hspace{.01in}}

\newcommand{\End}{\operatorname{End}\nolimits}

\newcommand{\id}{\operatorname{id}\nolimits}
\newcommand{\Id}{\operatorname{Id}\nolimits}

\newcommand{\op}{\operatorname{op}\nolimits}

\newcommand{\triv}{\operatorname{triv}}

\newcommand{\seg}{small extension groups\xspace}

\renewcommand{\subset}{\subseteq}
\makeatletter
\def\Biggg#1{{\hbox{$\left#1\vbox to25\p@{}\right.\n@space$}}}
\makeatother

\newcommand{\set}[2]{\left \{ \, {#1} \, \middle \vert \, {#2} \, \right \}}
\newcommand{\s}{\mbox{}\xspace}
\newcommand{\coloneqq}{:=}
\crefname{figure}{diagram}{diagrams}
\Crefname{figure}{Diagram}{Diagrams}

\makeatletter
 \newcommand{\xmapsfrom}[2][]{%
    \ext@arrow3095\leftarrowfill@{#1}{#2}\mapsfromchar
}
\makeatother

\makeatletter
\def\Axioms#1{\expandafter\@Axioms\csname c@#1\endcsname}
\def\@Axioms#1{\ifcase#1\or (EA1)\or (EA1${}^{\op}$)\or (EA2)\or (EA2${}^{\op}$)\fi}

\AddEnumerateCounter{\Axioms}{\@Axioms}{First}
\makeatother

\numberwithin{equation}{section}

\begin{document}
   \title{$n$-extension closed subcategories of $n$-exangulated categories}

   \author{Carlo Klapproth}\address{Department of Mathematics, Aarhus University, 8000 Aarhus C, Denmark} 
   \email{carlo.klapproth@math.au.dk}
   \keywords{$n$-exangulated category, $n$-extension closed, $n$-exact category, $(n+2)$-angulated category, higher homological algebra, extriangulated category, exact category, condition (WIC)}
   \subjclass{18G99 (Primary) 18G15, 18G80 (Secondary)}
   \maketitle
   
   \begin{abstract}
      Let $n$ be a positive integer.
      We show that an $n$-extension closed subcategory of an $n$-exangulated category naturally inherits an $n$-exangulated structure through restriction of the ambient $n$-exangulated structure.
      Furthermore, we show that a strong version of the Obscure Axiom holds for $n$-exangulated categories, where $n \geq 2$.
      This allows us to characterize $n$-exact categories as $n$-exangulated categories with monic inflations and epic deflations.
      We also show that for an extriangulated category condition (WIC), which was introduced by Nakaoka and Palu, is equivalent to the underlying additive category being weakly idempotent complete.
      We then apply our results to show that $n$-extension closed subcategories of an $n$-exact category are again $n$-exact.
      Furthermore, we recover and improve results of Klapproth and Zhou.
   \end{abstract}

   \section*{Introduction}
    Generalisation and abstraction are very useful tools as they allow us to understand an area of mathematics as a whole rather than locally.
   Furthermore, they are very efficient as proofs need to be carried out only once.
   To this end, Herschend, Liu and Nakaoka recently introduced the notion of $n$-exangulated categories, see \cite{HLN21} and \cite{HLN22}.
   These categories generalize $n$-exact and $(n+2)$-angulated categories simultaneously, in a similar way as extriangulated categories generalize exact and triangulated categories.
   These structures also allow us to compare the recently emerged field of higher homological algebra to classic homological algebra.

   Extension closed subcategories are an important part of homological algebra and representation theory and they appear naturally.
   For example the torsion and torsion free class of a torsion pair or the aisle and coaisle of a $t$-structure are extension closed subcategories.
   Therefore, we are interested in studying properties of them.
   It is well known that any extension closed subcategory of an exact category inherits an exact structure from the exact structure of the ambient category in a natural way, see for example Bühler \cite[Lemma 10.20]{Bueh10}.
   The same does not hold for triangulated categories, but the larger class of extriangulated categories is again closed under taking extension closed subcategories by \cite[Remark 2.18]{NP19}.
   The proof of this result is straightforward.

   For $n$-exangulated categories, where $n \in \NN_{\geq 2}$, the situation is more difficult.
   It was shown in He--Zhou \cite[Theorem 1.1]{HZ21} that an $n$-extension closed subcategory of a Krull--Schmidt $n$-exangulated category inherits an $n$-exangulated structure from the $n$-exangulated structure of the ambient category in a natural way.
   However, this is not completely satisfying, as a large class of categories, for example the category of finitely generated abelian groups $\mod \ZZ$ or its bounded derived category $\sD^{b}(\mod \ZZ)$, are not Krull-Schmidt.
   Indeed the Krull--Schmidt property restricts endomorphism rings to be semiperfect, see Krause \cite[Corollary 4.4]{Kra15}.
   Using a completely different method than employed in He--Zhou \cite{HZ21} we are able to show the following theorem in full generality.

   \begin{introtheorem}[{See \Cref{thm:extensionclosed}}]\label{introthmA}
      Suppose that $(\cC, \EE, \fs)$ is an $n$-exangulated category and $\cA \subset \cC$ is an $n$-extension closed additive subcategory.
      Then $\cA$ inherits an $n$-exangulated structure from $(\cC, \EE, \fs)$ in a natural way.
   \end{introtheorem}
   
    On the journey to this result technical obstacles need to be overcome.
   We prove the following useful theorem.

   \begin{introtheorem}[Strong Obscure Axiom, see \Cref{cor:obscure}]\label{introthm:obscure}
      Let $(\cC, \EE, \fs)$ be an $n$-exangulated category and $\cC$ weakly idempotent complete or $n \in \NN_{\geq 2}$.
      If $gf$ is an $\fs$-inflation \mbox{then so is $f$}.
   \end{introtheorem}

    It is now well known that for exact categories the strong version of the Obscure Axiom presented here is equivalent to the underlying additive category being weakly idempotent complete, see for example \cite[Propsition 7.6]{Bueh10}.
   However, when Quillen first defined exact categories, a weaker version of \Cref{introthm:obscure} with the assumption that $f$ admits a cokernel, was an important part of the original definition of exact categories, see Quillen \cite[Section 2]{Qui73}.
   It was discovered by Yoneda and later rediscovered by Keller that this axiom is a consequence of the other axioms for exact categories.

   It is remarkable that the strong Obscure Axiom holds for $n \in \NN_{\geq 2}$ without further assumption on the underlying category.
   The result can be interpreted as $n$-exangulated categories being more detached from the exact structure induced by the underlying additive category for $n \in \NN_{\geq 2}$ than they are for $n=1$.
   
   For extriangulated categories the strong Obscure Axiom corresponds to condition (WIC) introduced in \cite[Condition 5.8]{NP19}.
   This condition seems to be very important for several results in the theory of extriangulated categories, see for example \cite{ES22}, \cite{NP19} and \cite{WW22}.
   For example in \cite[Section 5]{NP19} it is used as a technical ingredient to show a bijection between Hovey twin cotorsion pairs and admissible model structures for extriangulated categories.
   It turns out that not only for exact but also for extriangulated categories the strong Obscure Axiom is equivalent to the underlying additive category being weakly idempotent complete.
   \begin{introprop}[See \Cref{prop:extri}]
      An extriangulated category satisfies condition \emph{(WIC)} if and only if it is weakly idempotent complete.
   \end{introprop}
    This is particularly interesting, since any extriangulated category can be weakly idempotent completed, see for example \cite[Theorem 3.31]{Msa22} for small extriangulated categories, or \Cref{rem:extri1} and \cite[Theorem 5.5]{KMS22}.
   Hence, any extriangulated category is a full subcategory of an extriangulated category where condition (WIC) is satisfied.

   Another application of the strong Obscure Axiom is to characterise $n$-exangulated categories which arise from $n$-exact categories completely by properties of their inflations and deflations.
   We prove the following theorem.
   \begin{introtheorem}[See \Cref{thm:correspondence}]\label{introcorr}
      For an additive category $\cC$ there is a one-to-one correspondence
      \begin{align*}
          \left\{\parbox{13em}{\centering $n$-exact structures $(\cC, \cX)$ with \seg}\right\} &\xleftrightarrow{\text{\emph{1:1}}} \frac{\left\{\parbox{17em}{\centering $n$-exangulated structures $(\cC, \EE, \fs)$ with\\ monic $\fs$-inflations and epic $\fs$-deflations}\right\}}{\left\{\parbox{17em}{\centering equivalences of $n$-exanuglated categories of the form $(\Id_{\cC}, \Gamma)$}\right\}}.
      \end{align*}
   \end{introtheorem}
    For the relationship of extriangulated and exact categories a similar result is well-known, see \cite[Example 2.13]{NP19} and \cite[Corollary 3.18]{NP19}.
   However, for $n$-exangulated categories, where $n \in \NN_{\geq 2}$, our strong Obscure Axiom is the missing ingredient to make the bijection constructed in \cite[Section 4.3]{HLN21} complete.

   The combination of \Cref{introthmA} and \Cref{introcorr} allows us to show the following theorem.
   \begin{introtheorem}[{See \Cref{cor:extensionclosed2}}]
      Suppose $(\cC, \cX)$ is an $n$-exact category with \seg and $\cA \subset \cC$ is an $n$-extension closed additive subcategory.
      Then $\cA$ inherits an $n$-exact structure in a natural way.
   \end{introtheorem}
    Furthermore, we improve results \cite[Theorem 1.2]{Zho22} and \cite[Theorem I]{Kla21} about $n$-extension closed subcategories of $(n+2)$-angulated categories, see \Cref{cor:thefirstcorollary} and \Cref{recoverkla}.
   
   \section{Conventions and notation}
    Throughout this paper we will use the notion of $n$-exangulated categories.
   We refer to \cite[Section 2]{HLN21} for the definition and \cite{HLN21} and \cite{HLN22} for an introduction.
   In \Cref{sec:nex} and \Cref{sec:nang} we consider $n$-exact categories.
   We refer to Jasso \mbox{\cite[Section 4]{Jas16}} for an introduction.
   In \Cref{sec:nang} we will study $(n+2)$-angulated categories.
   We refer to Geiss--Keller--Oppermann \cite[Section 2]{GKO13} for the definition. 
   
   For the rest of this section suppose we are given an arbitrary additive category $\cD$ and $n \in \NN_{\geq 1}$.
   We recall the following definitions.
   \begin{definition}
      A subcategory $\cB \subset \cD$ is called an \emph{additive subcategory} if $\cB \subset \cD$ is full and closed under finite direct sums.
      That means $\cB$ is closed under isomorphisms, $0 \in \cB$ and for $B, B' \in \cB$ the biproduct $B \oplus B'$ is in $\cB$.
   \end{definition}

    Notice that for the scope of this paper, we do \emph{not} require an additive subcategory to be closed under direct summands.
   Note, that we assume additive subcategories to be full. 

   \begin{remark}
   We will use cohomological degrees for complexes, i.e.\s the differentials go from lower to higher degree.
   On the other hand, we use homological notation, i.e.\s we use subscripts instead of superscripts, to avoid confusion with powers of morphisms.
   \end{remark}

   \begin{definition}[{\cite[Definition 2.7]{HLN21}}]
      Let $\C^{n+2}_\cD$ be the full subcategory of the category of complexes over $\cD$ consisting of all complexes concentrated in degrees $0, \dots, n+1$. 
      We denote objects in $\C^{n+2}_\cD$ by
      \[ X_\bullet \colon \hspace{1cm} \begin{tikzcd} X_0 \ar[r, "d^X_0"] & X_1 \ar[r, "d^X_1"] & \cdots \ar[r, "d^X_{n-1}"] & X_n \ar[r, "d^X_n"] & X_{n+1} \end{tikzcd}\]
      and morphisms in $\C_{\cD}^{n+2}(X_\bullet, Y_\bullet)$ by $f_\bullet = (f_0, f_1, \dots, f_n, f_{n+1}) \colon X_\bullet \to Y_\bullet$.
   \end{definition}
  
    We will use the following special complexes.

    \begin{definition}\label{def:trivialcomplex}
     For $X \in \cD$ and $i = 0, \dots, n$ let $\triv_i(X)_\bullet \in \C^{n+2}_{\cD}$ denote the complex
     \[\triv_i(X)_\bullet \colon \hspace{1cm} 
      \begin{tikzcd}[ampersand replacement=\&, column sep = small]
        0 \& \cdots \& 0 \& X \& X \& 0 \& \cdots \& 0
       \arrow[Rightarrow, no head, from=1-4, to=1-5]
        \arrow[from=1-3, to=1-4]
        \arrow[from=1-5, to=1-6]
       \arrow[from=1-6, to=1-7]
        \arrow[from=1-7, to=1-8]
        \arrow[from=1-1, to=1-2]
       \arrow[from=1-2, to=1-3]
      \end{tikzcd}\]
      with $\triv_i(X)_j = 0$ for $j \neq i, i+1$ and $\triv_i(X)_i = X = \triv_i(X)_{i+1}$ as well as $d^{\triv_i(X)}_i = \id_X$.
   \end{definition}

   \begin{definition}[{\cite[Definition 2.27]{HLN21}}]
      Let $f_\bullet \in \C^{n+2}_{\cD}(X_\bullet, Y_\bullet)$ be a morphism of complexes.
      If $A \coloneqq X_0 = Y_0$ and $f_0 = \id_A$ we denote by 
      \[M^f_\bullet \colon \hspace{1cm} \begin{tikzcd}[ampersand replacement=\&, column sep = huge, cramped]
        {X_1} \&[-1.8em] {X_2 \oplus Y_1} \& \cdots \&[+1em] {X_{n+1} \oplus Y_{n}} \&[-.2em] {Y_{n+1}}
        \arrow["{\left[\begin{smallmatrix}-d^X_1 \\ f_1\end{smallmatrix}\right]}", from=1-1, to=1-2]
        \arrow["{\left[\begin{smallmatrix}-d^X_{n} & 0 \\ f_{n} & d^Y_{n-1}\end{smallmatrix}\right]}", from=1-3, to=1-4]
        \arrow["{\left[\begin{smallmatrix}f_{n+1} & d^Y_n\end{smallmatrix}\right]}", from=1-4, to=1-5]
        \arrow["{\left[\begin{smallmatrix}-d^X_2 & 0 \\ f_2 & d^Y_1\end{smallmatrix}\right]}", from=1-2, to=1-3]
      \end{tikzcd}\]
      the \emph{mapping cone of $f_\bullet$}.
      Dually, if $C \coloneqq X_{n+1} = Y_{n+1}$ and $f_{n+1} = \id_C$ we denote by 
      \[N^f_\bullet \colon \hspace{1cm} \begin{tikzcd}[ampersand replacement=\&, column sep = huge, cramped]
        {X_0} \&[-1.8em] {X_1 \oplus Y_0} \& \cdots \&[+1em] {X_n \oplus Y_{n-1}} \&[-.2em] {Y_n}
        \arrow["{\left[\begin{smallmatrix}d^X_0 \\ f_0\end{smallmatrix}\right]}", from=1-1, to=1-2]
        \arrow["{\left[\begin{smallmatrix}d^X_1 & 0 \\ f_1 & -d^Y_0\end{smallmatrix}\right]}", from=1-2, to=1-3]
        \arrow["{\left[\begin{smallmatrix}d^X_{n-1} & 0 \\ f_{n-1} & -d^Y_{n-2}\end{smallmatrix}\right]}", from=1-3, to=1-4]
        \arrow["{\left[\begin{smallmatrix}f_n & -d^Y_{n-1}\end{smallmatrix}\right]}", from=1-4, to=1-5]
      \end{tikzcd}\]
      the \emph{mapping cocone of $f_\bullet$}.
   \end{definition}

   \begin{definition}[{\cite[Defintion 2.17]{HLN21}}]
      For $A,C \in \cD$ let $\smash{\C^{n+2}_{(\cD; A,C)}}$ be the subcategory of $\C^{n+2}_{\cD}$ where objects are complexes $X_\bullet$ with $X_0 = A$ and $X_{n+1} = C$ and where morphisms $f_\bullet \in \smash{\C^{n+2}_{(\cD; A, C)}(X_\bullet, Y_\bullet)}$ are morphisms $f_\bullet \in\C^{n+2}_{\cD}(X_\bullet, Y_\bullet)$ with $f_0 = \id_A$ and $f_{n+1} = \id_C$.
   \end{definition}
  
   \begin{definition}[{\cite[Definition 2.17]{HLN21}}]
      For $A,C \in \cD$ and $X_\bullet, Y_\bullet \in \C^{n+2}_{(\cD; A,C)}$ we say $f_\bullet, g_\bullet \in \smash{\C^{n+2}_{(\cD; A, C)}}(X_\bullet, Y_\bullet)$ are \emph{homotopy equivalent} if $f_\bullet - g_\bullet$ is zero homotopic as a morphism of complexes $X_\bullet \to Y_\bullet$.
   \end{definition}
    Homotopy equivalence induces an equivalence relation on morphisms and objects of $\smash{\C^{n+2}_{(\cD; A, C)}}$ for $A,C \in \cD$, see \cite[Definition 2.17]{HLN21}.
   We use a different notation than \cite{HLN21} to denote the equivalence class of an object in $\C^{n+2}_{(\cD;A,C)}$.

   \begin{notation}
      For any pair of objects $A,C \in \cD$ and any $X_\bullet \in \smash{\C^{n+2}_{(\cD; A, C)}}$ we denote by $[X_\bullet]_{\cD}$ the homotopy equivalence class of $X_\bullet$ in $\C^{n+2}_{(\cD,A,C)}$.
   \end{notation}

%
%
%
%
    Now, suppose we are additionally given an arbitrary biadditive functor $\GG \colon \cD^{\op} \times \cD \to \Ab$.

   \begin{notation}
      For any pair $C,A \in \cD$ the elements of $\GG(C,A)$ are called \emph{$\GG$-extensions}.
      We will denote by ${}_{A} 0_{C}$ the neutral element of $\GG(C,A)$.
      We will often simply write $0$ instead of ${}_A 0_C$.
      Furthermore, if there is no risk of confusion, we will denote 
      \[\text{$a_\ast \coloneqq \GG(C,a) \colon \GG(C,A) \to \GG(C, B)$ and $c^\ast \coloneqq \GG(c, A) \colon \GG(C,A) \to \GG(D, A)$}\] 
      for $A,B,C,D \in \cD$, $a \in \cD(A,B)$ and $c \in \cD(D, C)$
   \end{notation}
   
   \begin{definition}[{\cite[Definition 2.3]{HLN21}}]
      For $\GG$-extensions $\delta \in \GG(C,A)$ and $\rho \in \GG(D,B)$ a \emph{morphism of $\GG$-extensions $(a,c) \colon \delta \to \rho$} is a tuple consisting of a morphism $a \in \cD(A,B)$ and $c \in \cD(C,D)$ with $a_\ast \delta = c^\ast \rho$.
   \end{definition}

   \begin{definition}[{\cite[Definition 2.9]{HLN21}}]
      We define a category $\text{\AE}^{n+2}_{(\cD, \GG)}$ as follows. 
      \begin{enumerate}
         \item Objects are tuples $\langle X_\bullet, \delta \rangle$ consisting of $X_\bullet \in \C^{n+2}_{\cD}$ and $\delta \in \GG(X_{n+1}, X_0)$ satisfying $(d^X_0)_\ast \delta = 0$ and $(d^X_n)^\ast \delta = 0$.
         We call $\langle X_\bullet, \delta \rangle$ a \emph{$\GG$-attached complex} and denote
         \[ \langle X_\bullet, \delta \rangle \colon \hspace{1cm} \begin{tikzcd} X_0 \ar[r, "d^X_0"] & X_1 \ar[r, "d^X_1"] & \cdots \ar[r, "d^X_{n-1}"] & X_n \ar[r, "d^X_n"] & X_{n+1} \ar[r, dashed, "\delta"] &{}. \end{tikzcd}\]
         \item Morphisms $f_\bullet \in \smash{\text{\AE}^{n+2}_{(\cD, \GG)}}(\langle X_\bullet, \delta \rangle, \langle Y_\bullet, \rho \rangle)$ are morphisms $f_\bullet \in \smash{\C^{n+2}_{\cD}(X_\bullet, Y_\bullet)}$ such that $(f_0, f_{n+1}) \colon \delta \to \rho$ is a morphism of $\GG$-extensions.
         \item Composition is the same as in $\C^{n+2}_{\cD}$.
      \end{enumerate}
   \end{definition}
    We want to remark that $\text{\AE}^{n+2}_{(\cD, \GG)}$ is an additive category.

   \begin{notation}[{\cite[Definition 2.11]{HLN21}}]
      If there is no risk of confusion we denote
      \[\text{$\delta_\sharp \colon \cD(X,C) \to \GG(X,A),\, f \mapsto f^\ast \delta$ and $\delta^\sharp \colon \cD(A, X) \to \GG(C, X),\, g \mapsto g_\ast \delta$}\]
      for any $X \in \cD$ and $\delta \in \GG(C,A)$.
   \end{notation}

   \begin{definition}[{\cite[Definition 2.13]{HLN21}}]
      An object $\langle X_\bullet, \delta \rangle \in \text{\AE}^{n+2}_{(\cD,\GG)}$ is called an \emph{$n$-exangle} if for all $X \in \cD$ the sequences
      \[\begin{tikzcd}[ampersand replacement=\&, column sep = large]
        {\cD(X,X_0)} \& {\cD(X,X_1)} \& \cdots \& {\cD(X,X_{n+1})} \&[-2em] {\GG(X,X_0)}
        \arrow["{\cD(X,d^X_0)}", from=1-1, to=1-2]
        \arrow["{\cD(X,d^X_1)}", from=1-2, to=1-3]
        \arrow["{\cD(X,d^X_n)}", from=1-3, to=1-4]
        \arrow["{\delta_\sharp}", from=1-4, to=1-5]
      \end{tikzcd}\]
      and 
      \[\begin{tikzcd}[ampersand replacement=\&, column sep = large]
        {\cD(X_{n+1},X)} \& {\cD(X_n, X)} \& \cdots \& {\cD(X_0, X)} \&[-2em] {\GG(X_{n+1}, X)}
        \arrow["{\cD(d^X_n,X)}", from=1-1, to=1-2]
        \arrow["{\cD(d^X_{n-1},X)}", from=1-2, to=1-3]
        \arrow["{\cD(d^X_0,X)}", from=1-3, to=1-4]
        \arrow["{\delta^\sharp}", from=1-4, to=1-5]
      \end{tikzcd}\]
      are exact in $\Ab$. A morphism of $n$-exangles is just a morphism in $\text{\AE}^{n+2}_{(\cD, \GG)}$.
   \end{definition}
    In particular, if $\langle X_\bullet, \delta \rangle$ is an $n$-exangle then $d^X_{i+1}$ is a weak cokernel of $d^X_i$ for any $i =0, \dots, n-1$.
   Dually, $d^X_{i-1}$ is a weak kernel of $d^X_i$ for any $i=1,\dots,n$.

   \begin{definition}[{\cite[Definition 2.22]{HLN21}}]
      An \emph{exact realisation of $\GG$} is a map $\fr$ that assigns to each pair $A,C \in \cD$ and each $\delta \in \GG(A,C)$ a homotopy equivalence class $\fr(\delta)$ of an object in $\C^{n+2}_{(\cD,A,C)}$ such that the following axioms hold.
      \begin{enumerate}[label={(R\arabic*)}]
         \setcounter{enumi}{-1}
         \item For any objects $A,B,C,D \in \cD$, any $\GG$-extensions $\delta \in \GG(C, A)$ and $\rho \in \GG(D,B)$, any complexes $X_\bullet \in \cC^{n+2}_{(\cD; A,C)}$ with $[X_\bullet]_{\cD} = \fr(\delta)$ and $Y_\bullet \in \cC^{n+2}_{(\cD; B,D)}$ with $[Y_\bullet]_{\cD} = \fr(\rho)$ and any morphism $(a, c) \colon \delta \to \rho$ of $\GG$-extensions there exists an $f_\bullet \colon X_\bullet \to Y_\bullet$ with $f_0 = a$ and $f_{n+1} = c$.
         We call $f_\bullet$ a \emph{lift} of $(a,c)$.
         \item If $\fr(\delta) = [X_\bullet]_{\cD}$ for some $A,C \in \cD$, $\delta \in \GG(C, A)$ and $X_\bullet \in \C^{n+2}_{(\cD; A, C)}$ then $\langle X_\bullet, \delta \rangle$ is an $n$-exangle.
         \item For $A,C \in \cD$ we have $\fr({}_A 0_0) = [\triv_0(A)_\bullet]_{\cD}$ and $\fr({}_0 0_C) = [\triv_n(C)_\bullet]_{\cD}$.
      \end{enumerate}
      For any pair $A,C \in \cD$ and $\GG$-extension $\delta \in \GG(A,C)$ we call any $X_\bullet \in \C^{n+2}_{(\cD,A,C)}$ with $[X_\bullet] = \fr(\delta)$ an \emph{$\fr$-realisation} of $\delta$.
   \end{definition}

   \begin{definition}
      Suppose $\fr$ is an exact realisation of $\GG$. For any objects $A,C \in \cD$, any $\GG$-extension $\delta \in \GG(C,A)$ and any $\fr$-realisation $X_\bullet$ of $\delta$ we call
      \begin{enumerate}
         \item the $n$-exangle $\langle X_\bullet, \delta \rangle$ an \emph{$\fr$-distinguished $n$-exangle} and
         \item $X_\bullet$ an \emph{$\fr$-conflation}, $d^X_0$ an \emph{$\fr$-inflation} as well as $d^X_{n}$ an \emph{$\fr$-deflation}.
      \end{enumerate}
      A \emph{morphism of $\fr$-distinguished $n$-exangles} is just a morphism of $n$-exangles.
   \end{definition}

    In the following definition we divide \cite[Definition 2.32(EA1)]{HLN21} into two separate statements (EA1) and (EA1${}^{\op}$).
   \begin{definition}[{\cite[Definition 2.32]{HLN21}}]
      An \emph{$n$-exangulated category} is a triplet $(\cD, \GG, \fr)$ where $\cD$ is an additive category, $\GG \colon \cD^{\op} \times \cD \to \Ab$ is a biadditive functor and $\fr$ is an exact realisation of $\GG$ such that the following axioms hold.
      \begin{enumerate}[leftmargin = 4\parindent, label={\Axioms*}]
         \item If $f \in \cD(X,Y)$ and $g \in \cD(Y, Z)$ are $\fr$-inflations then so is $gf \in \cD(X,Z)$.\label{item:EA1} 
         \item If $f \in \cD(X,Y)$ and $g \in \cD(Y, Z)$ are $\fr$-deflations then so is $gf \in \cD(X,Z)$.\label{item:EA1op}
         \item For any morphism $c \in \cD(X_{n+1}, Y_{n+1})$ and any pair of $\fr$-distinguished $n$-exangles $\langle X_\bullet, c^* \rho \rangle$ and $\langle Y_\bullet, \rho \rangle$ with $A \coloneqq Y_0 = X_0$ there is a lift $f_\bullet \colon X_\bullet \to Y_\bullet$ of $(\id_A, c)$ such that $\langle M^f_\bullet, (d^X_0)_\ast \rho \rangle$ is $\fr$-distinguished.
               We call $f_\bullet$ a \emph{good lift} of $(\id_A, c)$.\label{item:EA2}
         \item For any morphism $a \in \cD(X_{0}, Y_{0})$ and any pair of $\fr$-distinguished $n$-exangles $\langle X_\bullet, \delta \rangle$ and $\langle Y_\bullet, a_\ast \delta \rangle$ with $C \coloneqq Y_{n+1} = X_{n+1}$ there is a lift $f_\bullet \colon X_\bullet \to Y_\bullet$ of $(a, \id_C)$ such that $\langle N^f_\bullet, (d^Y_{n})^\ast \rho \rangle$ is $\fr$-distinguished.
               We call $f_\bullet$ a \emph{good lift} of $(a, \id_C)$.\label{item:EA2op}
      \end{enumerate}
   \end{definition}

   \begin{remark}\label{rem:extri1}
      By \cite[Proposition 4.3]{HLN21} a triplet $(\cD, \GG, \fr)$ is a $1$-exangulated category if and only if it is an extriangulated category in the sense of \cite{NP19}.
      We therefore may use the term extriangulated category synonymously with the term $1$-exangulated category.
   \end{remark}

   \begin{definition}[{\cite[Definition 4.1]{HLN22}}]\label{def:exangulatedextclosed}
      An additive subcategory $\cB \subset \cD$ of an \mbox{$n$-exangulated} category $(\cD, \GG, \fr)$ is called \emph{$n$-extension closed} if for all $A,C \in \cB$ and $\delta \in \GG(C, A)$ there is an $\fr$-distinguished $n$-exangle $\langle X_\bullet, \delta \rangle$ with $X_i \in \cB$ for $i = 0, \dots, n+1$.
   \end{definition}
   
   \begin{remark}
         The notion of $1$-extension closed additive subcategories coincides with the notion of extension closed subcategories of \cite[Definition 2.17]{NP19} as any two extriangles realizing the same extension have isomorphic terms by \cite[Lemma 4.1]{HLN21} and additive subcategories are closed under isomorphisms.
    \end{remark}

    We recall the notion of an $n$-exangulated functor from Bennett-Tennenhaus--Shah.
   \begin{definition}[{\cite[Definition 2.32]{BTS21}}]
      Let $(\cD, \GG, \fr)$ and $(\cD', \GG', \fr')$ be $n$-exangulated categories.
      An \emph{$n$-exangulated functor} $(\sF, \Gamma) \colon (\cD, \GG, \fr) \to (\cD', \GG, \fr')$ is a tuple consisting of an additive functor $\sF \colon \cD \to \cD'$ and a natural transformation $\Gamma \colon \GG(-,-) \Rightarrow \GG'(\sF-, \sF-)$, such that $[X_\bullet]_{\cD} = \fr(\delta)$ implies $[\sF(X_\bullet)]_{\cD'} = \fr'(\Gamma_{C,A}(\delta))$ for all $A,C \in \cD$, $\delta \in \GG(A, C)$ and $X_\bullet \in \C^{n+2}_{(\cD;A,C)}$.

   \end{definition}

    Bennett-Tennenhaus--Haugland--Sandøy--Shah \cite[Definition 4.9]{BTHSS22} introduced the notion of $n$-exangulated equivalences.
   Using \cite[Proposition 4.11]{BTHSS22} we obtain the following equivalent definition.
   \begin{definition}
      A functor $(\sF, \Gamma) \colon (\cD, \GG, \fr) \to (\cD', \GG, \fr')$ of $n$-exangulated categories $(\cD, \GG, \fr)$ and $(\cD', \GG', \fr')$ is called an \emph{$n$-exangulated equivalence} if $\sF \colon \cD \to \cD'$ is an equivalence and $\Gamma \colon \GG(-,-) \Rightarrow \GG'(\sF-, \sF-)$ is a natural isomorphism.
   \end{definition}

    From now we assume the following global \Cref{stp:thesetup}, unless \emph{explicitly} stated otherwise.
   \begin{setup}\label{stp:thesetup}
   Suppose $n \in \NN_{\geq 1}$.
   Let $(\cC, \EE, \fs)$ be an $n$-exangulated category and $\cA \subset \cC$ be an $n$-extension closed additive subcategory.
   Let $\sJ_\cA \colon \cA \to \cC$ denote the canonical inclusion. 
   \end{setup}
    In the situation of \Cref{stp:thesetup} one can define a functor $\FF$ on $\cA$ and an exact realisation $\ft$.
   We will use the following notation for the rest of this paper.
   \begin{definition}[{\cite[Proposition 4.2]{HLN22}}] \label{def:induced}
   We define $\FF(-,-) \coloneqq \EE(\sJ_{\cA} -, \sJ_\cA -)$ to be the restriction of $\EE$.
   For $A, C \in \cA$ we define $(\Theta_{\cA})_{C,A} \colon \FF(C,A) \to \EE(C,A),\, \delta \mapsto \delta$ as the canonical inclusion.
   This yields a natural isomorphism $\Theta_{\cA} \colon \FF(-,-) \to \EE(\sJ_\cA-, \sJ_\cA-)$.
   For $A, C \in \cA$ and $\delta \in \FF(C,A)$ we define $\ft(\delta) \coloneqq [X_\bullet]_{\cA}$ where $\langle X_\bullet, \delta \rangle$ is an $\fs$-distinguished $n$-exangle with $X_i \in \cA$ for $i = 0, \dots, n+1$. 
   \end{definition}
    Notice that $\ft$ is well-defined since for any pair $A,C \in \cA$ and $X_\bullet, Y_\bullet \in \C^{n+2}_{(\cA; A, C)}$ we have $[X_\bullet]_{\cA} = [Y_\bullet]_{\cA}$ if and only if $[X_\bullet]_{\cC} = [Y_\bullet]_{\cC}$, as homotopy equivalence are preserved and reflected under $\sJ_{\cA}$, since $\cA \subset \cC$ is additive.

   Recall also that $\ft$ is an exact realisation of $\FF$ and that $(\cA, \FF, \ft)$ satisfies axioms \ref{item:EA2}, \ref{item:EA2op} by \cite[Propsition 4.2(1)]{HLN22}.
   We have the following important remark which we will make extensive use of. 
   \begin{remark}\label{rem:theremark}
      An $\FF$-attached complex $\langle X_\bullet, \delta \rangle$ with $\delta \in \FF(X_{n+1}, X_0)$ is a $\ft$-distinguished $n$-exangle if and only if $\langle \sJ_\cA(X_\bullet), (\Theta_\cA)_{X_{n+1}, X_0} (\delta) \rangle = \langle X_\bullet, \delta \rangle$ is an $\fs$-distinguished $n$-exangle with $X_i \in \cA$ for $i = 0, \dots, n+1$.
      Indeed, if $\langle X_\bullet, \delta \rangle$ is $\ft$-distinguished then $X_0, \dots, X_{n+1} \in \cA$ and $[X_\bullet]_{\cA} = [Y_\bullet]_{\cA}$ for an $\fs$-distinguished $n$-exangle $\langle Y_\bullet, \delta \rangle$ with $Y_0, \dots, Y_{n+1} \in \cA$, by definition.
      However, then $[X_\bullet]_{\cC} = [Y_\bullet]_{\cC} = \fs(\delta)$, since $\sJ_{\cA}$ preserves homotopy equivalences and hence $\langle X_\bullet, \delta \rangle$ is $\fs$-distinguished.
      On the other hand, if $\langle X_\bullet, \delta \rangle$ is $\fs$-distinguished and $X_0, \dots, X_{n+1} \in \cA$, then $\langle X_\bullet, \delta \rangle$ is $\ft$-distinguished, since $\ft$ is well-defined.
   \end{remark}

   \section{The Obscure Axiom}
    Recall \Cref{stp:thesetup} and \Cref{def:induced}.
   Before we can start, we need an easy but crucial lemma, which is similar to \cite[Corollary 3.4]{HLN21}.

   \begin{lemma}\label{lem:trivialsummand}
       Suppose $n \in \NN_{\geq 2}$.
       If $\left[\begin{smallmatrix} 0 & f \end{smallmatrix}\right]^\top \colon X_0 \to A \oplus X_1'$ is a $\ft$-inflation with $A,X_1' \in \cA$ then $f \colon X_0 \to X_1'$ is a $\ft$-inflation. 
   \end{lemma}
   \begin{proof}
      Suppose $\langle X_\bullet, \delta \rangle$ is a $\ft$-distinguished $n$-exangle with $X_1 = A \oplus X_1'$ and $d^X_0 = \left[\begin{smallmatrix} 0 & f \end{smallmatrix}\right]^\top$.
      We construct a commutative diagram 
      \[\begin{tikzcd}[ampersand replacement=\&, column sep = 5em, row sep = large]
         {X_0} \& {A \oplus X_1'} \& {X_2}. \\
            \& A
            \arrow["{d^X_0 = \left[\begin{smallmatrix}0 \\ f\end{smallmatrix}\right]}", from=1-1, to=1-2]
            \arrow["{\left[\begin{smallmatrix}g & h\end{smallmatrix} \right] \coloneqq d^X_1}", from=1-2, to=1-3]
            \arrow["{p \coloneqq \left[\begin{smallmatrix}\id_{A} & 0 \end{smallmatrix}\right]}", from=1-2, to=2-2]
            \arrow["{p'}", bend left=15, dotted, from=1-3, to=2-2]
         \end{tikzcd}\]      
      Let $p \coloneqq \left[\begin{smallmatrix} \smash{\id_{A}} & 0 \end{smallmatrix}\right] \colon A \oplus X_1' \to A$.
      Then $p d^X_0 = 0$ and because $d^X_1$ is a weak cokernel of $d^X_0$ there is a $p' \colon X_2 \to A$ with $p' d^X_1 = p$.
      Denote $d^X_1 \colon A \oplus X_1' \to X_2$ by $\left[\begin{smallmatrix}g & h\end{smallmatrix}\right]$.
      Then $p' g = \id_A$ and $p' h = 0$.
      Hence, $p'$ is a retraction with section $g$ and $e \coloneqq gp'$ and $e' \coloneqq \id_{X_2} - g p'$ are orthogonal idempotents.
        The $n$-exangle $\langle \triv_2(A)_\bullet, 0 \rangle$ is $\fs$-distinguished using $n \in \NN_{\geq 2}$ and \cite[Proposition 2.14]{Hau21}.
        Notice that this crucially depends on $n \in \NN_{\geq 2}$ as for $n = 1$ there is not enough space to define $\triv_2(A)$, compare \Cref{def:trivialcomplex}.
      Hence, the $n$-exangle $\langle X''_\bullet, \delta' \rangle \coloneqq \langle X_\bullet \oplus \triv_2(A)_\bullet, \delta \oplus 0 \rangle$ 
      \[\begin{tikzcd}[ampersand replacement=\&, cramped]
        {X_0} \& {A \oplus X_1'} \& {X_2 \oplus A} \&[+1em] {X_3 \oplus A} \& {X_4} \& \cdots \& {X_{n+1}} \& {} 
        \arrow["{\left[\begin{smallmatrix}g & h \\ 0 & 0\end{smallmatrix}\right]}", from=1-2, to=1-3]
        \arrow["\delta'", dashed, from=1-7, to=1-8]
        \arrow["{\left[\begin{smallmatrix}d^X_3 & 0\end{smallmatrix}\right]}", from=1-4, to=1-5]
        \arrow["{\left[\begin{smallmatrix}0 \\ f\end{smallmatrix}\right]}", from=1-1, to=1-2]
        \arrow["{\left[\begin{smallmatrix}d^X_2 & 0 \\ 0 & \id_A \end{smallmatrix}\right]}", from=1-3, to=1-4]
        \arrow["{d^X_4}", from=1-5, to=1-6]
        \arrow["{d^X_n}", from=1-6, to=1-7]
      \end{tikzcd}\]
      is $\fs$-distinguished, by \Cref{rem:theremark} and \cite[Corollary 2.26(2)]{HLN21}.
      It is easy to check that \Cref{dgr:biproduct}
      \begin{figure}[H]\begin{tikzcd}[ampersand replacement=\&, row sep = large, cramped]
        {X_0} \& {X_1'} \& {X_2} \& {X_3 \oplus A} \& {X_4} \& \cdots \&[-.5em] {X_{n+1}} \& {} \\
        {X_0} \& {A \oplus X_1'} \& {X_2 \oplus A} \& {X_3 \oplus A} \& {X_4} \& \cdots \& {X_{n+1}} \& {} \\
        0 \& A \& A \& 0 \& 0 \& \cdots \& 0 \& {},
        \arrow["{\delta'}", dashed, from=1-7, to=1-8]
        \arrow["0", dashed, from=3-7, to=3-8]
        \arrow["{d^X_n}", from=1-6, to=1-7]
        \arrow[from=3-6, to=3-7]
        \arrow["{d^X_4}", from=1-5, to=1-6]
        \arrow[from=3-5, to=3-6]
        \arrow[Rightarrow, no head, from=1-5, to=2-5]
        \arrow[shift right=1, from=2-5, to=3-5]
        \arrow[shift right=1, from=3-5, to=2-5]
        \arrow["f", from=1-1, to=1-2]
        \arrow["h", from=1-2, to=1-3]
        \arrow["{\left[\begin{smallmatrix}d^X_2 e' \\ p'\end{smallmatrix}\right]}", from=1-3, to=1-4]
        \arrow[from=3-1, to=3-2]
        \arrow[Rightarrow, no head, from=3-2, to=3-3]
        \arrow[from=3-3, to=3-4]
        \arrow["{d^{X''}_3}", from=1-4, to=1-5]
        \arrow[shift left=1, from=3-4, to=3-5]
        \arrow[Rightarrow, no head, from=2-1, to=1-1]
        \arrow["{\left[\begin{smallmatrix}\id_A & 0\end{smallmatrix}\right]}"', shift right=1, from=2-2, to=3-2]
        \arrow["{\left[\begin{smallmatrix}\id_A \\ 0\end{smallmatrix}\right]}"', shift right=1, from=3-2, to=2-2]
        \arrow["{\left[\begin{smallmatrix}p' & 0\end{smallmatrix}\right] }"', shift right=1, from=2-3, to=3-3]
        \arrow["{\left[\begin{smallmatrix}g \\ 0\end{smallmatrix}\right]}"', shift right=1, from=3-3, to=2-3]
        \arrow[shift right=1, from=2-4, to=3-4]
        \arrow[shift right=1, from=3-4, to=2-4]
        \arrow[Rightarrow, no head, from=1-4, to=2-4]
        \arrow["{\left[\begin{smallmatrix}e' \\ p' \end{smallmatrix}\right]}", shift left=1, from=1-3, to=2-3]
        \arrow["{\left[\begin{smallmatrix}e' & g \end{smallmatrix}\right]}", shift left=1, from=2-3, to=1-3]
        \arrow["{\left[\begin{smallmatrix}0 & \id_{\smash{X_1'}}\end{smallmatrix}\right]}", shift left=1, from=2-2, to=1-2]
        \arrow["{\left[\begin{smallmatrix} 0 \\ \id_{\smash{X_1'}}\end{smallmatrix}\right]}", shift left=1, from=1-2, to=2-2]
        \arrow[shift right=1, from=2-1, to=3-1]
        \arrow[shift right=1, from=3-1, to=2-1]
        \arrow[Rightarrow, no head, from=1-7, to=2-7]
        \arrow[shift right=1, from=2-7, to=3-7]
        \arrow[shift right=1, from=3-7, to=2-7]
        \arrow["{d^{X''}_0}", from=2-1, to=2-2]
        \arrow["{d^{X''}_1}", from=2-2, to=2-3]
        \arrow["{d^{X''}_2}", from=2-3, to=2-4]
        \arrow["{d^{X''}_3}", from=2-4, to=2-5]
        \arrow["{d^{X''}_4}", from=2-5, to=2-6]
        \arrow["{d^{X''}_n}", from=2-6, to=2-7]
        \arrow["{\delta'}", from=2-7, to=2-8, dashed]
      \end{tikzcd}
      \caption{Biproduct diagram in $\text{\AE}_{(\cC, \EE)}^{n+2}$.}\label{dgr:biproduct}\end{figure}
      \noindent where the middle row is $\langle X''_\bullet, \delta' \rangle$, is a biproduct diagram in the additive category $\text{\AE}_{(\cC, \EE)}^{n+2}$, see \Cref{calc:biproduct}.
      By \cite[Proposition 3.3]{HLN21} for $(\cC, \EE, \fs)$ this means that the upper row of \Cref{dgr:biproduct} is an $\fs$-distinguished $n$-exangle $\langle X'_\bullet, \delta' \rangle$.
      All terms of $\langle X'_\bullet, \delta' \rangle$ are in $\cA$.
      This shows that $\langle X'_\bullet, \delta' \rangle$ is a $\ft$-distinguished $n$-exangle, by \Cref{rem:theremark}.
      Hence $f$ is a $\ft$-inflation.
   \end{proof}

    The proof of \Cref{lem:trivialsummand} depends on $n \in \NN_{\geq 2}$.
   However, \Cref{lem:trivialsummand} still holds for $n = 1$ if $\cA$ is weakly idempotent complete.
   Indeed, we can then just remove a trivial summand $\langle \triv_1(A)_\bullet, 0 \rangle$ from $\langle X_\bullet, \delta \rangle$. 
   For the case where $\cA = \cC$ this has been shown by Tattar, see \cite[Lemma II.1.43]{Tat22}.
   We provide a proof for convenience of the reader.
   \begin{lemma}\label{lem:weakidempotent}
      Suppose $\cA$ is weakly idempotent complete and $n=1$.
      If $\left[\begin{smallmatrix} 0 & f \end{smallmatrix}\right]^\top \colon X_0 \to A \oplus X_1'$ is a $\ft$-inflation with $A,X_1' \in \cA$ then $f \colon X_0 \to X_1'$ is a $\ft$-inflation. 
   \end{lemma}
   \begin{proof}
      As $\left[\begin{smallmatrix} 0 & f \end{smallmatrix}\right]^\top \colon X_0 \to A \oplus X_1'$ is a $\ft$-inflation there is a $\ft$-distinguished $1$-exangle $\langle X_\bullet, \delta \rangle$ with $X_1 = A \oplus X'_1$ and $d^X_0 = \left[\begin{smallmatrix} 0 & f \end{smallmatrix}\right]^\top$.
      We construct the following diagram
      \[\begin{tikzcd}[ampersand replacement=\&, column sep = 7em, row sep = large]
        {X_0} \& {A \oplus X_1'} \& {X_2} \\
        \& A \&
        \arrow["{d^X_1}", from=1-2, to=1-3]
        \arrow["{d^X_0 = \left[\begin{smallmatrix}0 \\ f\end{smallmatrix}\right]}", from=1-1, to=1-2]
        \arrow["{i \coloneqq \left[\begin{smallmatrix}\id_A \\ 0\end{smallmatrix}\right]}", shift left=1, from=2-2, to=1-2]
        \arrow["{p \coloneqq \left[\begin{smallmatrix}\id_{A} & 0\end{smallmatrix}\right]}", shift left=1, from=1-2, to=2-2]
        \arrow["{p'}", shift right=1, bend left=15, dotted, from=1-3, to=2-2]
      \end{tikzcd}\]
      in $\cA$.
      Let $p \coloneqq \left[\begin{smallmatrix} \smash{\id_{A}} & 0 \end{smallmatrix}\right] \colon A \oplus X_1' \to A$.
      Then $p d^X_0 = 0$ and because $d^X_1$ is a weak cokernel of $d^X_0$ there is $p' \colon X_2 \to A$ with $p' d^X_1 = p$.
      Now, $i \coloneqq \left[\begin{smallmatrix} \smash{\id_{A}} & 0 \end{smallmatrix}\right]^\top \colon A \to A \oplus X_1'$ is a section for $p$. 
      Hence, $d^X_1 i$ is a section for the retraction $p'$ and $e \coloneqq d_1^X i p' \in \End_\cC(X_2)$ is a split idempotent in $\cA$.
      As $\cA$ is weakly idempotent complete there is a splitting of $e' \coloneqq \id_{X_2} - e$, say with retraction $q' \colon X_2 \to X_2'$ and section $j' \colon X_2' \to X_2$ such that $e' = j'q'$ and $X_2' \in \cA$.
      Put $\delta' \coloneqq (j')^\ast \delta$, $q \coloneqq \left[\begin{smallmatrix} 0 & \id_{\smash{X_1'}} \end{smallmatrix}\right]$ and $j \coloneqq \left[\begin{smallmatrix} 0 & \id_{\smash{X_1'}} \end{smallmatrix}\right]^\top$. 
      It is easy to check that \Cref{dgr:biproduct2}
      \begin{figure}[H]
        \begin{tikzcd}
            X_0 \arrow[d, equal] \arrow[r, "f"]                                                                & X_1' \arrow[d, "j"', shift right] \arrow[r, "q'd^X_1j"]               & X_2'   \arrow[d, "j'"', shift right]  \arrow[r, "\delta'", dashed]               & {} \\
            X_0 \arrow[d, shift left] \arrow[r, "{\left[\begin{smallmatrix} 0 \\ f \end{smallmatrix}\right]}"] & A \oplus X_1'               \arrow[d, "p", shift left] \arrow[u, "q"', shift right] \arrow[r, "d^X_1"] & X_2 \arrow[d, "p'", shift left] \arrow[u, "q'"', shift right] \arrow[r, "\delta", dashed]           & {} \\
            0 \arrow[r] \ar[u, shift left]                                                                     & A \arrow[u, "i", shift left] \arrow[r, equal]                                          & A \arrow[u, "d_1^Xi", shift left]                                   \arrow[r, "{}_0 0_A", dashed]                      & {},
        \end{tikzcd}
      \caption{Biproduct diagram in $\text{\AE}_{(\cC, \EE)}^{3}$.}\label{dgr:biproduct2}
      \end{figure}
      \noindent where the middle row is the $1$-exangle $\langle X_\bullet, \delta \rangle$, is a biproduct diagram in the additive category $\smash{\text{\AE}_{(\cC, \EE)}^{3}}$, see \Cref{calc:biproduct2}.
      By \cite[Proposition 3.3]{HLN21} for $(\cC, \EE, \fs)$ this means that the upper row of \Cref{dgr:biproduct2} is an $\fs$-distinguished $n$-exangle $\langle X'_\bullet, \delta' \rangle$.
      All terms of $X'_\bullet$ are in $\cA$, so $\langle X', \delta' \rangle$ is $\ft$-distinguished, by \Cref{rem:theremark}.
      Hence, $f$ is a $\ft$-inflation.
   \end{proof}

   \begin{lemma}\label{lem:reversemappingcone}
      If $g \colon X_0 \to X_1$ is a $\ft$-inflation and $f \colon X_0 \to A$ is a morphism with $A \in \cA$ then $\left[\begin{smallmatrix} f & g \end{smallmatrix}\right]^\top \colon X_0 \to A \oplus X_1$ is a $\ft$-inflation. 
   \end{lemma}
   \begin{proof}
      We can complete the $\ft$-inflation $g$ to a $\ft$-distinguished $n$-exangle $\langle X_\bullet, \delta \rangle$ with $d^X_0 = g$.
      As $A, X_{n+1} \in \cA$ there is a $\ft$-distinguished $n$-exangle $\langle Y_\bullet, f_\ast \delta \rangle$ with $Y_0 = A$ and $Y_{n+1} = X_{n+1}$.
      The solid morphisms $f$ and $\id_{X_{n+1}}$ in the diagram
      \[\begin{tikzcd}[ampersand replacement = \&]
         X_0 \ar[r, "g"] \ar[d, "{f}"]  \& X_1 \ar[r] \ar[d, dotted] \& \cdots \ar[r] \& X_n \ar[r] \ar[d, dotted] \& X_{n+1} \ar[d, equal] \ar[r, "\delta", dashed] \& {} \\ 
         A \ar[r]                \& Y_1 \ar[r]                \& \cdots \ar[r] \& Y_n \ar[r]                \& X_{n+1} \ar[r, dashed, "f_\ast \delta"]    \& {}
      \end{tikzcd}\]
      form a morphism of $\FF$-extensions $(f, \id_{X_{n+1}}) \colon  \delta \to f_\ast \delta$.
      Since $(\cA, \FF, \ft)$ satisfies axiom \ref{item:EA2op} we can find a good lift $f_\bullet \colon \langle X_\bullet, \delta \rangle \to \langle Y_\bullet, f_\ast \delta \rangle $ of $(f, \id_{X_{n+1}}) \colon \delta \to f_\ast \delta$ such that the mapping cocone
      \[\langle N^{f}_\bullet, (d^Y_n)^\ast \delta \rangle \colon \hspace{.5cm} 
      \begin{tikzcd}[ampersand replacement = \&, cramped]
         X_0 \ar[r, "{\left[\begin{smallmatrix} g \\ f \end{smallmatrix}\right]}"] \& X_1 \oplus A \ar[r, "d^{N^{f}}_1"] \&[-.5em] X_2 \oplus Y_1 \ar[r] \&[-1.5em] \cdots \ar[r] \&[-1.5em] X_n \oplus Y_{n-1} \ar[r] \&[-1.5em] Y_{n} \ar[r, dashed, "(d^Y_n)^\ast \delta"] \&[+.5em] {}
      \end{tikzcd}\]
      of $f_\bullet$ is $\ft$-distinguished.
      Now there is an isomorphism 
      \[s \coloneqq \left[\begin{smallmatrix} 0 & \id_{A} \\ \id_{X_1} & 0  \end{smallmatrix}\right] \colon X_1 \oplus A \to A \oplus X_1 \]
      so $\langle N^{f}_\bullet, (d^Y_n)^\ast\delta \rangle$ is isomorphic to
      \[\langle N_\bullet, \delta \rangle \colon \hspace{1cm} 
      \begin{tikzcd}[ampersand replacement = \&, cramped]
          X_0 \ar[r, "{\left[\begin{smallmatrix} f \\ g \end{smallmatrix}\right]}"] \& A \oplus X_1 \ar[r, "d^{N^{f}}_1 s^{-1}"] \&[+.5em] X_2 \oplus Y_1 \ar[r] \&[-1.5em] \cdots \ar[r] \&[-1.5em] X_n \oplus Y_{n-1} \ar[r] \&[-1.5em] Y_{n} \ar[r, dashed, "(d^Y_n)^\ast \delta"] \&[+.5em] {}.
      \end{tikzcd}\]
      By \cite[Corollary 2.26(2)]{HLN21} this is a $\ft$-distinguished $n$-exangle. 
      The result follows.
      \end{proof}
      
      \begin{proposition}[Relative Obscure Axiom]\label{prop:obscure}
         Suppose $\cA$ is weakly idempotent complete or $n \in \NN_{\geq 2}$.
         Let $f \colon X \to Y$ and $g \colon Y \to Z$ be two morphisms with $Y \in \cA$.
         If $gf \colon X \to Z$ is a $\ft$-inflation, then so is $f$.
      \end{proposition}
      \begin{proof}
         We have $Y \in \cA$.
            Therefore, $\left[\begin{smallmatrix} f & gf \end{smallmatrix}\right]^{\top} \colon X \to Y \oplus Z$ is a $\ft$-inflation by applying \Cref{lem:reversemappingcone} to the $\ft$-inflation $gf \colon X \to Z$ and the morphism $f \colon X \to Y$.
         Hence, there is a $\ft$-distinguished $n$-exangle $\langle X_\bullet, \delta \rangle$ with $X_0 = X$, $X_1 = Y \oplus Z$ and $d^{X}_0 = \smash{\left[\begin{smallmatrix} f & gf \end{smallmatrix}\right]^\top}$.
         Consider the isomorphism 
         \[s \coloneqq \left[\begin{smallmatrix} 0 & \id_Z \\ \id_Y & 0 \end{smallmatrix}\right] \left[\begin{smallmatrix} \id_Y & 0 \\ -g & \id_Z\end{smallmatrix}\right] \colon Y\oplus Z \to Z \oplus Y.\]
         This isomorphism satisfies
         \[s \left[\begin{smallmatrix} f \\ gf \end{smallmatrix}\right] = \left[\begin{smallmatrix} 0 & \id_Z \\ \id_Y & 0 \end{smallmatrix}\right] \left[\begin{smallmatrix} \id_Y & 0 \\ -g & \id_Z\end{smallmatrix}\right] \left[\begin{smallmatrix} f \\ gf \end{smallmatrix}\right] = \left[\begin{smallmatrix} 0 & \id_Z \\ \id_Y & 0 \end{smallmatrix}\right] \left[\begin{smallmatrix} f \\ 0 \end{smallmatrix}\right] =  \left[\begin{smallmatrix} 0 \\ f \end{smallmatrix}\right] \]
         Using $s$ and \cite[Corollary 2.26(2)]{HLN21}, the $\ft$-distinguished $n$-exangle $\langle X_\bullet, \delta \rangle$ gives rise to an $\ft$-distinguished $n$-exangle
         \[\begin{tikzcd}[ampersand replacement=\&]
             X \&[-.5em] {Z \oplus Y} \&[+1em] {X_2} \&[-1em] \cdots \&[-1em] {X_n} \&[-1em] {X_{n+1}} \& {}.
             \arrow["{\left[\begin{smallmatrix} 0 \\ f \end{smallmatrix}\right]}", from=1-1, to=1-2]
             \arrow["{d^{X}_1 s^{-1}}", from=1-2, to=1-3]
            \arrow[from=1-3, to=1-4]
            \arrow[from=1-4, to=1-5]
            \arrow[from=1-5, to=1-6]
            \arrow["{\delta}", dashed, from=1-6, to=1-7]
         \end{tikzcd}\]
         Hence, $\left[\begin{smallmatrix} 0 & f\end{smallmatrix}\right]^{\top}$ is a $\ft$-inflation.
         Notice, $Y,Z \in \cA$.
         Hence, $f$ is a $\ft$-inflation by \Cref{lem:weakidempotent} if $\cA$ is weakly idempotent complete and $n=1$ and by \Cref{lem:trivialsummand} if $n \in \NN_{\geq 2}$.
      \end{proof}

       It is remarkable that, for $n \in \NN_{\geq 2}$, we do not need to assume that $\cC$ is weakly idempotent complete for the following to hold.

      \begin{corollary}[Strong Obscure Axiom]\label{cor:obscure}
         Suppose $\cC$ is weakly idempotent complete or $n \in \NN_{\geq 2}$.
         Let $f \colon X \to Y$ and $g \colon Y \to Z$ are two morphisms. 
         If $gf \colon X \to Z$ is an $\fs$-inflation, then so is $f$.
      \end{corollary}
      \begin{proof}
         This follows immediately from \Cref{prop:obscure}.
      \end{proof}
    
       Indeed, the converse of \Cref{cor:obscure} is true as well. We recall the following definition.

      \begin{definition}[{\cite[Condition 5.8]{NP19}}]\label{def:WIC}
         An extriangulated category $(\cC, \EE, \fs)$ satisfies \emph{condition} (WIC) if for any two morphisms $f \colon X \to Y$ and $g \colon Y \to Z$ the following hold.
         \begin{enumerate}
            \item If $gf$ is a $\fs$-inflation, then $f$ is a $\fs$-inflation.
            \item If $gf$ is a $\fs$-deflation, then $g$ is a $\fs$-deflation.
         \end{enumerate}
      \end{definition}

      \begin{proposition} \label{prop:extri}
         An extriangulated category satisfies condition \emph{(WIC)} if and only if it is weakly idempotent complete.
      \end{proposition}
      \begin{proof}
         That a weakly idempotent complete extriangulated category satisfies condition (WIC) follows from \Cref{cor:obscure} and its dual using \Cref{rem:extri1}.
         That an extriangulated category which satisfies condition (WIC) is weakly idempotent complete follows from \cite[Proposition 3.33]{Msa22} or \cite[Corollary II.1.41]{Tat22}.
      \end{proof}

      \section{\texorpdfstring{$n$}{n}-extension closed subcategories of \texorpdfstring{$n$}{n}-exangulated categories}
       Recall \Cref{stp:thesetup} and \Cref{def:induced}.
      By \cite[Proposition 4.2(1)]{HLN22}, we know that $\ft$ is an exact realisation for $\FF$ and that $(\cA, \FF, \ft)$ satisfies axioms \ref{item:EA2} and \ref{item:EA2op}.
      To show that $(\cA, \FF, \ft)$ is $n$-exangulated we only need to show that $(\cA, \FF, \ft)$ satisfies axioms \ref{item:EA1} and \ref{item:EA1op}, by \cite[Proposition 4.2(2)]{HLN22}.
      We will show that $\ft$-inflations are closed under composition, the remaining axiom \ref{item:EA1op} follows dually.
      
      If $f \colon X_0 \to X_1$ and $g \colon X_1 \to Y_1$ are $\ft$-inflations then $gf$ is an $\fs$-inflation by \Cref{rem:theremark} and axiom \ref{item:EA1} for $(\cC, \EE, \fs)$.
      By completing the inflations $f$, $gf$ and $g$ to distinguished $n$-exangles, we may obtain the solid morphism of \Cref{dgr:inflations}
      \begin{figure}[H]
         \begin{tikzcd}
            X_0 \arrow[r, "f"] \arrow[d, equal]               & X_1 \arrow[r, "d_1^X"] \arrow[d, "g"]   & X_2 \arrow[r] \arrow[d, "\phi_2", dotted] & \cdots \arrow[r] & X_n \arrow[r] \arrow[d, "\phi_n", dotted] & X_{n+1} \arrow[d, "\phi_{n+1}", dotted] \arrow[r, "\delta", dashed]   & {} \\
            Y_0 \arrow[d, "f"] \arrow[r, "gf"]                & Y_1 \arrow[d, equal] \arrow[r]          & Y_2 \arrow[d, dotted] \arrow[r]           & \cdots \arrow[r] & Y_n \arrow[r] \arrow[d, dotted]           & Y_{n+1} \arrow[d, dotted] \arrow[r, "\rho", dashed]                   & {} \\
            Z_0 \arrow[r, "g"]                                & Z_1 \arrow[r]                           & Z_2 \arrow[r]                             & \cdots \arrow[r] & Z_n \arrow[r]                             & Z_{n+1} \arrow[r, "\gamma", dashed]                                   & {}
         \end{tikzcd}
         \caption{The $n$-exangles arising from $\ft$-inflations $f$ and $g$.}
         \label{dgr:inflations}
      \end{figure}
      \noindent such that the upper and lower row are $\ft$-distinguished $n$-exangles and the middle row is an $\fs$-distinguished $n$-exangle.
      Our plan is to replace the object $Y_{n+1}$ by an object $Y' \in \cA$ and $\rho \in \EE(Y_{n+1}, Y_0)$ by an $\FF$-extension $\varepsilon \in \FF(Y', Y_0)$, see \Cref{lem:extensionfix}.
      Then we want to realise $\varepsilon$ by a $\ft$-distinguished $n$-exangle and replace the $\ft$-inflation of this $n$-exangle by $gf$ using the relative Obscure Axiom.
   
      \begin{lemma} \label{lem:extensionfix}
         Suppose we are given the solid morphisms of \Cref{dgr:inflations} such that the upper and lower row, respectively, form $\ft$-distinguished $n$-exangles $\langle X_\bullet, \delta \rangle$ and $\langle Z_\bullet, \gamma \rangle$, and such that the middle row forms an $\fs$-distinguished $n$-exangle $\langle Y_\bullet, \rho \rangle$.
         Then there is an object $Y' \in \cA$ and morphisms $s \colon Y' \to Y_{n+1}$ and $t \colon Y_{n+1} \to Y'$ such that $(st)^\ast \rho = \rho$.
      \end{lemma}
      \begin{proof}
         It follows from \cite[Proposition 3.6(2)]{HLN21} applied in $(\cC, \EE, \fs)$ that there is a morphism $\phi_\bullet \colon \langle X_\bullet, \delta \rangle \to \langle Y_\bullet, \rho \rangle$ with $\phi_0 = \id_{X_0}$ and $\phi_1 = g$ such that the mapping cone $\langle M^\phi_\bullet, f_\ast \rho \rangle$ of $\phi_\bullet$ is $\fs$-distinguished. 
         Notice that $Z_0 = X_1, Y_1 = Z_1, X_2 \in \cA$.
         By \Cref{lem:reversemappingcone} for the $\ft$-inflation $g \colon X_1 \to Y_1$ and the morphism $-d_1^X \colon X_1 \to X_2$ we have that $\left[ \begin{smallmatrix} -d^X_1 & g \end{smallmatrix}\right]^\top \colon X_1 \to X_2 \oplus Y_1$ is a $\ft$-inflation.
         Hence, there is a $\ft$-distinguished $n$-exangle $\langle Z'_\bullet, \gamma' \rangle$ with $Z'_0 = X_1$, $Z'_1 = X_2 \oplus Y_1$ and $d^{Z'}_0 = \left[ \begin{smallmatrix} -d^X_1 & g \end{smallmatrix}\right]^\top$.
         We obtain the solid morphisms of a diagram 
         \[\begin{tikzcd}[ampersand replacement = \&, cramped, row sep = 2.5em]
            X_1 \arrow[d, equal] \arrow[r, "{\left[\begin{smallmatrix} -d^X_1 \\ g \end{smallmatrix}\right]}"] \&[+.5em] X_2 \oplus Y_1 \arrow[d, equal] \arrow[r] \&[-1.5em] Z'_2 \arrow[r] \arrow[d, "s'_2", dotted, shift left]            \&[-1.5em] \cdots \arrow[r] \&[-1.5em] Z'_n \arrow[r] \arrow[d, "s'_n", dotted, shift left]                                                                                           \&[+2em] Z'_{n+1} \arrow[d, "s'_{n+1}", dotted, shift left] \arrow[r, "\gamma'", dashed]   \& {} \\
            X_1 \arrow[r, "{\left[\begin{smallmatrix} -d^X_1 \\ g \end{smallmatrix}\right]}"]                  \&        X_2 \oplus Y_1 \arrow[r]                  \&         X_3 \oplus Y_2 \arrow[r] \arrow[u, "t'_2", dotted, shift left]  \&         \cdots \arrow[r] \&         X_{n+1} \oplus Y_n \arrow[r, "{\left[\begin{smallmatrix} \phi_{n+1} & d^Y_n \end{smallmatrix}\right]}"] \arrow[u, "t'_n", dotted, shift left]  \&       Y_{n+1} \arrow[u, "t'_{n+1}", dotted, shift left] \arrow[r, "f_\ast \rho", dashed]  \& {}
         \end{tikzcd}\]
         where the upper row is the $\ft$-distinguished $n$-exangle $\langle Z'_\bullet, \gamma' \rangle$ and the lower row is the $\fs$-distinguished $n$-exangle $\langle M^\phi_\bullet, f_\ast \rho \rangle$.
         By \cite[Propostion 3.6(1)]{HLN21}, this gives rise to morphisms $s'_\bullet \colon \langle Z'_\bullet, \gamma' \rangle \to \langle M^\phi_\bullet, f_\ast \rho \rangle$ and $t'_\bullet \colon \langle M^\phi_\bullet, f_\ast \rho \rangle \to \langle Z'_\bullet, \gamma' \rangle$ with $s'_0 = \id_{X_1} = t'_0$ and  $s'_1 = \id_{X_2 \oplus Y_1} = t'_1$.
         This implies that $(\id_{X_1}, s'_{n+1}) \colon \gamma' \to f_\ast \rho$ and $(\id_{X_1}, t'_{n+1}) \colon f_\ast \rho \to \gamma'$ are morphisms of $\EE$-extensions. 
         Hence, $(\id_{X_1}, s'_{n+1} t'_{n+1}) \colon f_\ast \rho \to f_\ast \rho$ is also a morphism of $\EE$-extensions.
         Therefore,
         \[(\id_{Y_{n+1}} - s'_{n+1} t'_{n+1})^\ast f_\ast \rho = f_\ast \rho - (s'_{n+1} t'_{n+1})^ \ast f_\ast \rho = f_\ast \rho - (\id_{X_1})_\ast f_\ast \rho = 0\] 
         holds. Because the sequence
         \[\begin{tikzcd}[column sep = large, ampersand replacement = \&]
            \cC(Y_{n+1}, X_{n+1} \oplus Y_n) \ar[r, "{\cC(Y_{n+1}, \left[\begin{smallmatrix} \phi_{n+1} & d^Y_n \end{smallmatrix}\right])}"] \&[+4em] \cC(Y_{n+1}, Y_{n+1}) \ar[r, "(f_\ast \rho)_\sharp"] \& \EE(Y_{n+1}, X_1)
         \end{tikzcd}\]
         is exact, there is a morphism $\left[ \begin{smallmatrix} h & h' \end{smallmatrix}\right]^\top \colon Y_{n+1} \to X_{n+1} \oplus Y_{n}$ with 
         \[\id_{Y_{n+1}} - s'_{n+1} t_{n+1}' = \left[\begin{smallmatrix} \phi_{n+1} & d^Y_n \end{smallmatrix}\right] \left[ \begin{smallmatrix} h \\ h' \end{smallmatrix}\right] = \phi_{n+1} h + d^Y_n h'.\]
         
         Now we define $Y' \coloneqq X_{n+1} \oplus Z'_{n+1}$ and $s \coloneqq \left[\begin{smallmatrix} \phi_{n+1} & s'_{n+1} \end{smallmatrix}\right] \colon  X_{n+1} \oplus Z'_{n+1} \to Y_{n+1}$ as well as $\smash{t \coloneqq \left[\begin{smallmatrix} h & t'_{n+1} \end{smallmatrix}\right]^\top \colon  Y_{n+1} \to X_{n+1}} \oplus Z'_{n+1} $.
         We claim that these are the desired morphisms.
         Indeed, $\id_{Y_{n+1}} - st = \id_{Y_{n+1}} - s'_{n+1} t'_{n+1} - \phi_{n+1} h = d_n^Y h'$.
         Therefore, we obtain $(\id_{Y_{n+1}} - st)^\ast \rho = (d_n^Y h')^\ast \rho = (h')^\ast (d_n^Y)^\ast \rho = 0$ since already $(d_n^Y)^\ast \rho = 0$, as all $n$-exangles are $\EE$-attached complexes.
         Since $\langle X_\bullet, \delta \rangle$ and $\langle Z'_\bullet, \gamma' \rangle$ were $\ft$-distinguished, we have $Y' = X_{n+1} \oplus Z'_{n+1} \in \cA$ and the result follows. 
      \end{proof}

       We are ready to prove that $\ft$-inflations are closed under composition.

      \begin{lemma}\label{lem:comp}
         If $f \colon X \to Y$ and $g \colon Y \to Z$ are $\ft$-inflations, then so is $gf \colon X \to Z$. 
      \end{lemma}
      \begin{proof}
         If $n = 1$ then \cite[Remark 2.18]{NP19} and \Cref{rem:extri1} imply that the triplet $(\cA, \FF, \ft)$ is $1$-exangulated.
         The \namecref{lem:comp} follows from \ref{item:EA1} in this case.

         Let $n \in \NN_{\geq 2}$.
         Define $X_0 \coloneqq X \eqqcolon Y_0$, $X_1 \coloneqq Y \eqqcolon Z_0$ and $Y_1 \coloneqq Z \eqqcolon Z_1$.
         Since $f$ and $g$ are $\ft$-inflations, hence $\fs$-inflations, we know that $gf$ is a $\fs$-inflation by \ref{item:EA1} for $(\cC, \EE, \fs)$.
         This shows that we can construct the solid morphisms of \Cref{dgr:inflations} such that the upper row and lower row, respectively, are $\ft$-distinguished $n$-exangles $\langle X_\bullet, \delta \rangle$ and $\langle Z_\bullet, \gamma \rangle$ and such that the middle row is an $\fs$-distinguished $n$-exangle $\langle Y_\bullet, \rho \rangle$.
         By \Cref{lem:extensionfix} there is an object $Y' \in \cA$ and morphisms $s \colon Y' \to Y_{n+1}$ and $t \colon Y_{n+1} \to Y'$ with $(st)^\ast \rho = \rho$.
         We have $s^\ast \rho \in \EE(Y', Y_0)$.
         Since $\cA$ is $n$-extension closed there is a $\ft$-distinguished $n$-exangle $\langle Y'_\bullet, s^\ast \rho \rangle$.
         We obtain the solid morphisms of a commutative diagram
         \[\begin{tikzcd}[ampersand replacement=\&]
            {Y_0} \& {Y_1} \& {Y_2} \& \cdots \& {Y_n} \& {Y_{n+1}} \& {} \\
            {Y'_0} \& {Y'_1} \& {Y'_2} \& \cdots \& {Y'_n} \& Y' \& {} 
            \arrow[from=2-1, to=2-2, "d^{Y'_0}"]
            \arrow[from=1-2, to=1-3]
            \arrow[from=1-3, to=1-4]
            \arrow[from=1-4, to=1-5]
            \arrow[from=1-5, to=1-6]
            \arrow[Rightarrow, no head, to=1-1, from=2-1]
            \arrow[shift left=0, dotted, to=2-5, from=1-5, "t_n"]
            \arrow[shift left=0, dotted, to=2-3, from=1-3, "t_2"]
            \arrow[shift left=0, dotted, to=2-2, from=1-2, "t_1"]
            \arrow[shift right=0, from=1-1, to=1-2, "gf"]
            \arrow[shift right=0, from=2-2, to=2-3]
            \arrow[shift right=0, from=2-3, to=2-4]
            \arrow[shift right=0, from=2-4, to=2-5]
            \arrow[shift right=0, from=2-5, to=2-6]
            \arrow["t", shift left=1, to=2-6, from=1-6]
            \arrow["\rho", shift right=0, dashed, from=1-6, to=1-7]
            \arrow["{s^\ast \rho}", shift right=0, dashed, from=2-6, to=2-7]
         \end{tikzcd}\]
         where the upper row is the $\fs$-distinguished $n$-exangle $\langle Y_\bullet, \rho \rangle$ and the lower row is the $\ft$-distinguished $n$-exangle $\langle Y'_\bullet, s^\ast \rho \rangle$.
         Since $(st)^\ast \rho = \rho$, the morphism $(\id_{Y_0}, t) \colon \rho \to s^\ast \rho$ is a morphism of $\EE$-extension and hence can be lifted to a morphism $t_\bullet \colon \langle Y_\bullet, \rho \rangle \to \langle Y'_\bullet, s^\ast \rho \rangle$ of $n$-exangles. 
         This gives us 
         $d^{Y'}_0 = t_1 gf$.
         Since $d^{Y'}_0 = t_1 (gf)$ is a $\ft$-inflation and $Y_1 \in \cA$, \Cref{prop:obscure}
         shows that $gf$ is a $\ft$-inflation.
         \end{proof} 

          The following theorem proves \cite[Theorem 1.1]{HZ21} in a more general setting.
         \begin{theorem}\label{thm:extensionclosed}
            Suppose that $(\cC, \EE, \fs)$ is an $n$-exangulated category with an $n$-extension closed additive subcategory $\cA \subset \cC$.
            Then $(\cA, \FF, \ft)$ is an $n$-exangulated category and $(\sJ_\cA, \Theta_\cA) \colon (\cA, \FF, \ft) \rightarrow (\cC, \EE, \fs)$ is a fully faithful $n$-exangulated functor.
         \end{theorem}
         \begin{proof}
            The first part follows from \cite[Proposition 4.2(2)]{HLN22}, \Cref{lem:comp} and its dual. 
            The second part is clear by definition of $(\sJ_{\cA}, \Theta_{\cA})$ and \Cref{rem:theremark}.
         \end{proof}

      \begin{remark}
         Notice that $(\cA, \FF, \ft)$ is an $n$-exangulated subcategory of $(\cC, \EE, \fs)$ in the sense of Haugland \cite[Definition 3.7]{Hau21}.
      \end{remark}

    \section{\texorpdfstring{$n$}{n}-extension closed subcategories of \texorpdfstring{$n$}{n}-exact categories}\label{sec:nex}
     Throughout this section we also assume that $\cD$ is an additive category.
    We recall the following definition.

    \begin{definition}[{\cite[Defintion 2.4]{Jas16}}]
      An object $X_\bullet \in \C^{n+2}_{\cD}$ is called an \emph{$n$-exact sequence} if for all $X \in \cD$ the sequences
      \[\begin{tikzcd}[ampersand replacement=\&, column sep = large]
        {0} \&[-2em] {\cD(X,X_0)} \& {\cD(X,X_1)} \& \cdots \& {\cD(X,X_{n+1})}
         \arrow[from=1-1, to=1-2]
        \arrow["{\cD(X,d^X_0)}", from=1-2, to=1-3]
        \arrow["{\cD(X,d^X_1)}", from=1-3, to=1-4]
        \arrow["{\cD(X,d^X_n)}", from=1-4, to=1-5]
      \end{tikzcd}\]
      and 
      \[\begin{tikzcd}[ampersand replacement=\&, column sep = large]
        { 0} \&[-2em] {\cD(X_{n+1},X)} \& {\cD(X_n, X)} \& \cdots \& {\cD(X_0, X)}
         \arrow[from=1-1, to=1-2]
        \arrow["{\cD(d^X_n,X)}", from=1-2, to=1-3]
        \arrow["{\cD(d^X_{n-1},X)}", from=1-3, to=1-4]
        \arrow["{\cD(d^X_0,X)}", from=1-4, to=1-5]
      \end{tikzcd}\]
      are exact in $\Ab$.
   \end{definition}

   \begin{notation}[{\cite[Definition 4.12]{HLN21}}]
      We denote by $\smash{\Lambda^{n+2}_{(\cD; A, C)}}$ the class of all homotopy equivalence classes of $n$-exact sequences in $\smash{\C^{n+2}_{(\cD;A,C)}}$.
   \end{notation}

       We recall the construction of $n$-exangulated categories from $n$-exact categories defined in \cite[Section 4.3]{HLN21}.
      Suppose that $(\cD, \cX)$ is an $n$-exact category in the sense of \cite[Definition 4.2]{Jas16}.
      For any pair $A, C \in \cD$ we define a class 
      \begin{equation}\label{eqn:Gdefinition} \GG_\cX(C, A) \coloneqq \set{[X_\bullet]_{\cD} \in \Lambda^{n+2}_{(\cD; A,C)}}{X_\bullet \in \cX} \end{equation} 
      as in \cite[Defintion 4.24]{HLN21}. 
      This does not have to be a set in general.
      \begin{definition}
         We say an $n$-exact category $(\cD, \cX)$ has \emph{\seg} if $\GG_\cX(C, A)$ as defined in (\ref{eqn:Gdefinition}) is a set for all $A, C \in \cD$.
      \end{definition} 
       We recall the following construction from \cite[Definition 4.16]{HLN21}.
      For $A, B, C \in \cD$, $[X_\bullet]_{\cD} \in \GG_\cX(C, A)$ and $a \colon A \to B$ we define 
      \begin{equation}\label{eqn:Gmorphism} \GG_\cX(C, a)\big([X_\bullet]_{\cD}\big) \coloneqq [Y_\bullet]_{\cD}\end{equation} 
      by picking an $Y_\bullet \in \cX \cap \C^{n+2}_{(\cD;B,C)}$ such that there is a morphism $f_\bullet \in \C^{n+2}_{\cD} (X_\bullet, Y_\bullet)$ with $f_0 = a$ and $f_{n+1} = \id_C$ making the solid part of the diagram
      \begin{equation}\label{eqn:n-pullback}\begin{tikzcd}[ampersand replacement=\&]
          {A} \& {X_1} \& \cdots \& {X_{n-1}} \& X_n \& C \\
          {B} \& {Y_2} \& \cdots \& {Y_{n-1}} \& X_n \& C
        \arrow["{a}", from=1-1, to=2-1]
        \arrow["{d^Y_0}", from=2-1, to=2-2]
        \arrow["{d^X_0}", from=1-1, to=1-2]
        \arrow["{f_1}", from=1-2, to=2-2]
        \arrow["{d^Y_1}", from=2-2, to=2-3]
        \arrow["{d^X_1}", from=1-2, to=1-3]
        \arrow["{d^Y_{n-2}}", from=2-3, to=2-4]
        \arrow["{d^X_{n-2}}", from=1-3, to=1-4]
        \arrow["{f_{n-1}}", from=1-4, to=2-4]
        \arrow["{d^Y_{n-1}}", from=2-4, to=2-5]
        \arrow["f_n", from=1-5, to=2-5]
        \arrow["{d^X_{n-1}}", from=1-4, to=1-5]
        \arrow["{d^X_{n}}", from=1-5, to=1-6, dotted]
        \arrow["{d^Y_{n}}", from=2-5, to=2-6, dotted]
        \arrow[no head, from=1-6, to=2-6, dotted, shift left=1pt]
        \arrow[no head, from=1-6, to=2-6, dotted, shift right=1pt]
      \end{tikzcd}\end{equation}     
      an $n$-pushout diagram as defined in \cite[Definition 2.11]{Jas16}.
      Such a $Y_\bullet$ exists by using \cite[Definition 4.2]{Jas16} and \cite[Proposition 4.8]{Jas16} and the assignment (\ref{eqn:Gmorphism}) is well-defined by \cite[Proposition 4.18]{HLN21}.
      Dually, we can define $\GG_\cX(c, A)\big([X_\bullet]_\cD \big)$ for $A,C,D \in \cD$, $[X_\bullet]_\cD \in \GG_\cX(C,A)$ and $c \colon D \to C$.

      If $(\cD, \cX)$ has additionally \seg, a bifunctor $\GG_\cX \colon \cD^{\op} \times \cD \to \Ab$ can be defined this way, see \cite[Definition 4.24, Lemma 4.26 and Proposition 4.32]{HLN21}.
      
      Recall that the additive structure on $\GG_\cX (C,A)$ for $C,A \in \cD$ is defined through Baer sums as follows.
      For $[X_\bullet]_{\cD}, [Y_\bullet]_{\cD} \in \GG_\cX(C, A)$ we have $[X_\bullet \oplus Y_\bullet]_{\cD} \in \GG_\cX(C \oplus C, A \oplus A)$ by \cite[Proposition 4.6]{Jas16} and we can define $[X_\bullet]_\cD + [Y_\bullet]_\cD \coloneqq \GG_{\cX}(\Delta_C, \nabla_A)([X_\bullet \oplus Y_\bullet]_\cD)$, where $\Delta_C = \left[\begin{smallmatrix} \id_C & \id_C \end{smallmatrix} \right]^{\top} \colon C \to C \oplus C$ is the diagonal and $\nabla_A = \left[\begin{smallmatrix} \id_A & \id_A \end{smallmatrix} \right] \colon A \oplus A \to A$ is the codiagonal, see \cite[Definition 4.28]{HLN21}.

      \begin{notation}\label{def:exact->exangulated}
          For an $n$-exact category $(\cD, \cX)$ with \seg we denote by $\GG_{\cX}$ the functor constructed above and by $\fr_\cX$ the assigment $\fr_\cX(\delta) = [X_\bullet]_{\cD}$ for $A,C \in \cD$ and $\delta = [X_\bullet]_{\cD} \in \GG_\cX(C,A)$.
      \end{notation}
      
      \begin{proposition}\label{prop:exact->exangulated}
         If $(\cD,\cX)$ is an $n$-exact category with \seg then $(\cD, \GG_{\cX}, \fr_{\cX})$ is an $n$-exangulated category with monic $\fr_{\cX}$-inflations and epic $\fr_{\cX}$-deflations.
      \end{proposition}
      \begin{proof}
         The proof is given in \cite[Propsition 4.34]{HLN21} and \cite[Remark 4.35]{HLN21}.
      \end{proof}

      \begin{definition}
         We say an $n$-exangulated category $(\cD, \GG, \fr)$ is \emph{$n$-exact} if there exists an $n$-exact structure $\cX \subset \smash{\C^{n+2}_{\cD}}$ on $\cD$ and an equivalence of $n$-exangulated categories $(\Id_\cD, \Gamma) \colon  (\cD, \GG, \fr) \to (\cD, \GG_\cX, \fr_\cX)$.
      \end{definition}

       Conversely, we can construct $n$-exact categories from $n$-exangulated categories using \cite[Propsition 4.37]{HLN21} and the strong Obscure Axiom.
      \begin{notation}[{\cite[Lemma 4.36]{HLN21}}]\label{def:exangulated->exact}
         For an $n$-exangulated category $(\cC, \EE, \fs)$, denote by $\cX_{(\EE,\fs)}$ the class of all $\fs$-conflations.
      \end{notation}
      \begin{proposition} \label{prop:exangulated->exact}
         Suppose $(\cC, \EE, \fs)$ is an $n$-exangulated category such that all $\fs$-inflations are monic and all $\fs$-deflations are epic, then $(\cC, \cX_{(\EE, \fs)})$ is an $n$-exact category.
      \end{proposition}
      \begin{proof}
         For $n = 1$ this follows from \Cref{rem:extri1} and \cite[Corollary 3.18]{NP19}.
         For $n \in \NN_{\geq 2}$ this follows from \cite[Proposition 4.37(2)]{HLN21} because the two conditions (a) and (b) of \cite[Proposition 4.37(2)]{HLN21} are satisfied by \Cref{cor:obscure} and its dual.
      \end{proof}

      Showing that the construction of \Cref{prop:exact->exangulated} and \Cref{prop:exangulated->exact} are inverse to each other relies on the following \Cref{lem:natural-transformation,lem:exactidentity}.
      We need some setup.

         Let $(\cC, \EE, \fs)$ be an $n$-exangulated category in which all $\fs$-inflations are monic and all $\fs$-deflations are epic.
         Then $(\cC, \cX_{(\EE, \fs)})$ is an $n$-exact structure by \Cref{prop:exangulated->exact}.
         We obtain a class $\GG_{\cX_{(\EE, \fs)}}(C,A)$ for $C,A \in \cC$ through the assignment (\ref{eqn:Gdefinition}).
         We define a map 
         \[\Gamma_{(C,A)} \colon \EE(C, A) \to \GG_{\cX_{(\EE, \fs)}}(C, A),\, \delta \mapsto \fs(\delta)\] 
         for $C, A \in \cC$, which is bijective by \cite[Lemma 4.36(3)]{HLN21}.
         In particular, $(\cC, \cX_{(\EE, \fs)})$ has \seg by the Axiom of Replacement using that $\EE(C, A)$ is a set.
         Hence, $(\cC, \GG_{\cX_{(\EE, \fs)}}, \fr_{\cX_{(\EE, \fs)}})$ is an $n$-exangulated category by \Cref{prop:exact->exangulated}.
      \begin{lemma}\label{lem:natural-transformation}
         Under the above assumptions the following hold.
         \begin{enumerate}
            \item\label{item:natural1} For $A,B,C \in \cC$, $a \colon A \to B$ and $\delta \in \EE(C,A)$ we have 
               \[\Gamma_{(C,B)}\big(\EE(C,a)(\delta)\big) = \GG_{\cX_{(\EE, \fs)}}(C, a)\big(\Gamma_{(C,A)}(\delta)\big).\]
            \item\label{item:natural2} For $A,C,D \in \cC$, $c \colon D \to C$ and $\delta \in \EE(C,A)$ we have 
               \[\Gamma_{(D,A)}\big(\EE(c,A)(\delta)\big) = \GG_{\cX_{(\EE, \fs)}}(c, A)\big(\Gamma_{(C,A)}(\delta)\big).\]
            \item\label{item:additive} For $A,C \in \cC$ and $\delta, \rho \in \EE(C,A)$ we have 
               \[ \Gamma_{(C,A)}(\delta + \rho) = \Gamma_{(C,A)}(\delta) + \Gamma_{(C,A)}(\rho).\]
         \end{enumerate}
         Hence, $\Gamma \colon \EE \Rightarrow \GG_{\cX_{(\EE, \fs)}}$ is a natural isomorphism and
         \[(\Id_\cC, \Gamma) \colon (\cC, \EE, \fs) \to (\cC, \GG_{\cX_{(\EE, \fs)}}, \fr_{\cX_{(\EE, \fs)}}) \] 
         is an $n$-exangulated equivalence.
      \end{lemma}
      \begin{proof}
         (\ref{item:natural1}):\; Let $\rho \coloneqq \EE(C,a)(\delta)$ and $\langle X_\bullet, \delta \rangle, \langle Y_\bullet, \rho \rangle$ be $\fs$-distinguished $n$-exangles.
         Notice that $(a, \id_C) \colon \delta \to \rho$ is a morphism of $\EE$-extensions.
         There exists a lift $f_\bullet \colon X_\bullet \to Y_\bullet$ with $f_0 = a$ and $f_{n+1} = \id_C$.
         By the dual of \cite[Lemma 4.36(1)]{HLN21}, this means that the solid part of (\ref{eqn:n-pullback}) is an $n$-pushout diagram.
         We have $X_\bullet \in \cX_{(\EE, \fs)} \cap \C^{n+2}_{(\cC;A,C)}$ and  $Y_\bullet \in \cX_{(\EE, \fs)} \cap \C^{n+2}_{(\cC;B,C)}$. 
         Hence, 
         \[\Gamma_{(C,B)} \big(\EE(C,a)(\delta)\big) = \fs(\rho) =  [Y_\bullet]_{\cC} = \GG_{\cX_{(\EE, \fs)}}(C, a)\big([X_\bullet]_{\cC}\big) = \GG_{\cX_{(\EE, \fs)}}(C, a)\big(\Gamma_{C,A}(\delta)\big)  \]
         by definition of $\GG_{\cX_{(\EE, \fs)}}(C, a)$. In the same way (\ref{item:natural2}) can be shown.
      
         (\ref{item:additive}):\; Let $\delta, \rho \in \EE(C, A)$ for $A,C \in \cC$.
         Then we have $\delta + \rho = \EE(\Delta_C, \nabla_A)(\delta \oplus \rho)$, where $\Delta_C \coloneqq \left[\begin{smallmatrix} \id_C & \id_C \end{smallmatrix} \right]^\top \colon C \to C \oplus C$ and $\nabla_A = \left[\begin{smallmatrix} \id_A & \id_A \end{smallmatrix} \right] \colon A \oplus A \to A$ as mentioned in \cite[Definition 2.6]{HLN21}.
         We have
         \[ \Gamma_{(A,C)} (\delta + \rho) = \Gamma_{(C,A)}\big(\EE(\Delta_C, \nabla_A)(\delta \oplus \rho)\big) = \GG_{\cX_{(\EE, \fs)}}(\Delta_C, \nabla_A)\big(\Gamma_{(C\oplus C, A \oplus A)}(\delta \oplus \rho)\big)\]
         by using (\ref{item:natural1}) and (\ref{item:natural2}).
         Let $X_\bullet$ be an $\fs$-realisation of $\delta$ and $Y_\bullet$ be an $\fs$-realisation of $\rho$.
         Then $[X_\bullet \oplus Y_\bullet]_{\cC} = \fs(\delta \oplus \rho)$, by \cite[Proposition 3.3]{HLN21}.
         Hence,
         \begin{align*} \GG_{\cX_{(\EE, \fs)}}(\Delta_C, \nabla_A)\big(\Gamma_{(C\oplus C, A \oplus A)}(\delta \oplus \rho)\big) &= \GG_{\cX_{\EE, \fs}}(\Delta_C, \nabla_A)\big([X_\bullet \oplus Y_\bullet]_{\cC}) \\
                                                                                                                                 &= [X_\bullet]_{\cC} + [Y_\bullet]_{\cC} \\ 
                                                                                                                                 &= \Gamma_{(C,A)}(\delta) + \Gamma_{(C,A)}(\rho)
         \end{align*}
         as addition in $\GG_{\cX_{(\EE, \fs)}}(C,A)$ is defined through Baer sums. 
         Therefore, (\ref{item:additive}) holds.

         Finally, $\Gamma \colon \EE \Rightarrow \GG_{\cX_{(\EE, \fs)}}$ is a natural isomorphism of functors $\cC^{\op} \times \cC \to \Ab$ by (\ref{item:natural1}), (\ref{item:natural2}) and (\ref{item:additive}). 
         It is clear that $\fs(\delta) = \fr_{\cX_{(\EE, \fs)}}(\Gamma_{(C,A)}(\delta))$ for $A,C \in \cC$ and $\delta \in \EE(C, A)$, by definition. 
         Hence, $(\Id_\cC, \Gamma) \colon (\cC, \EE, \fs) \to (\cC, \GG_{\cX_{(\EE, \fs)}}, \fr_{\cX_{(\EE, \fs)}})$ is an $n$-exangulated equivalence.
      \end{proof}
         Suppose that $(\cD, \cX)$ is an $n$-exact category with \seg. 
         The $n$-exangulated category $(\cD, \GG_{\cX}, \fr_{\cX})$, as defined in \Cref{prop:exact->exangulated}, has monic $\fr_{\cX}$-inflations and epic $\fr_{\cX}$-deflations.
         Therefore, an $n$-exact category $(\cD, \cX_{(\GG_{\cX}, \fr_{\cX})})$ can be defined using \Cref{prop:exangulated->exact}.
      \begin{lemma}\label{lem:exactidentity}
         Under the above assumptions we have $\cX = \cX_{(\GG_\cX, \fr_\cX)}$.
      \end{lemma}
      \begin{proof}
         Let $X_\bullet \in \cX$ and $\delta \coloneqq [X_\bullet]_{\cD} \in \GG_{\cX}(X_{n+1}, X_0)$.
         Then $\fr_{\cX}(\delta) = [X_\bullet]_{\cD}$, by definition.
         Hence, $X_\bullet$ is an $\fr_\cX$-conflation.
         This means $X_\bullet \in \cX_{(\GG_{\cX}, \fr_{\cX})}$, again by definition.

         Conversely, let $X_\bullet \in \cX_{(\GG_{\cX}, \fr_{\cX})}$.
         Then $X_\bullet$ is an $\fr_\cX$-conflation.
         By definition of $\GG_\cX$ and $\fr_\cX$ this means there is a $Y_\bullet \in \cX$ such that $[Y_\bullet]_\cD = [X_\bullet]_\cD$.
         This implies the existence of an equivalence $X_\bullet \to Y_\bullet$ of $n$-exact sequences in the sense of \cite[Definition 2.9]{Jas16}.
         By \cite[Definition 4.2]{Jas16} the class $\cX$ is closed under weak isomorphisms  and hence $X_\bullet \in \cX$.
      \end{proof}
       We can now summarize \cite[Section 4.3]{HLN21} and show that the two constructions given are inverse to each other.

      \begin{theorem}\label{thm:correspondence}
         \Cref{prop:exact->exangulated} and \ref{prop:exangulated->exact} induce a one-to-one correspondence
         \begin{align*}
            \left\{\parbox{13em}{\centering $n$-exact structures $(\cD, \cX)$ with \seg}\right\} &\xleftrightarrow{\text{\emph{1:1}}} \frac{\left\{\parbox{17em}{\centering $n$-exangulated structures $(\cD, \GG, \fr)$ with\\ monic $\fr$-inflations and epic $\fr$-deflations}\right\}}{\left\{\parbox{17em}{\centering equivalences of $n$-exanuglated categories of the form $(\Id_{\cD}, \Gamma)$}\right\}} \\
                                              (\cD, \cX) &\xmapsto{\phantom{\text{\emph{1:1}}}} (\cD, \GG_\cX, \fr_{\cX}) \\
                                              (\cD, \cX_{(\GG, \fr)}) &\xmapsfrom{\phantom{\text{\emph{1:1}}}} (\cD, \GG, \fr).
         \end{align*}
      \end{theorem}
      \begin{proof}
         The map from left to right is well-defined by \Cref{prop:exact->exangulated}.

         For any $n$-exangulated equivalence $(\Id_\cD, \Gamma) \colon (\cD, \GG, \fr) \to (\cD, \GG', \fr')$ we have a natural isomorphism $\Gamma \colon \GG \Rightarrow \GG'$. 
         Moreover, any $\fr$-realisation of any $\GG$-extension is an $\fr'$-realisation of its image under $\Gamma$ as $(\Id_\cD, \Gamma)$ is an $n$-exangulated functor.
         Hence, the classes of conflations $\cX_{(\GG, \fr)}$ and $\cX_{(\GG', \fr')}$ coincides and the map from right to left is well-defined.

         The \namecref{thm:correspondence} follows now from \Cref{lem:natural-transformation} and \Cref{lem:exactidentity}.
      \end{proof}
      
      \begin{corollary}\label{cor:nexact}
         For any $n$-exangulated category $(\cC, \EE, \fs)$ the following are equivalent.
         \begin{enumerate}
            \item $(\cC, \EE, \fs)$ is $n$-exact.\label{item:nexact}
            \item Every $\fs$-inflation is monic and every $\fs$-deflation is epic.\label{item:monicepic}
         \end{enumerate}
      \end{corollary}

      \begin{definition}\label{def:n-exactextclosed}
         An additive subcategory $\cB \subset \cD$ of an $n$-exact category $(\cD, \cX)$ is called \emph{$n$-extension closed} if for all $X_\bullet \in \cX$ with $X_0, X_{n+1} \in \cB$ there exists a $Y_\bullet \in \C^{n+2}_{\cB} \cap \cX$ with $[X_\bullet]_{\cD} = [Y_\bullet]_{\cD}$.
      \end{definition}
 
       The two notions of $n$-extension closed given in \Cref{def:exangulatedextclosed,def:n-exactextclosed} coincide.
      \begin{lemma}\label{lem:extclosedcoincide1}
         Suppose $(\cD, \cX)$ is an $n$-exact category with an additive subcategory $\cB \subset \cD$.
         Then $\cB$ is $n$-extension closed in $(\cD, \cX)$ if and only if it is $n$-extension closed in $(\cD, \GG_\cX, \fr_\cX)$.
      \end{lemma}
      \begin{proof}
          We show only that if $\cB$ is $n$-extension closed in $(\cD, \cX)$ in the sense of \Cref{def:n-exactextclosed} then $\cB$ is $n$-extension closed in $(\cD, \GG_\cX, \fr_\cX)$ in the sense of \Cref{def:exangulatedextclosed}, the reverse statement follows similarly.
          Let $\delta \in \GG_\cX(C,A)$ for $C,A \in \cB$.
          We have $\cX = \cX_{(\GG_\cX, \fr_\cX)}$ by \Cref{thm:correspondence} and hence $X_\bullet \in \cX$ for any $X_\bullet \in \C^{n+2}_{\cD}$ with $[X_\bullet] = \fr_\cX(\delta)$.
          By \Cref{def:n-exactextclosed} we can pick $Y_\bullet \in \C^{n+2}_{\cB}$ with $[Y_\bullet] = [X_\bullet] = \fs(\delta)$.
      \end{proof}

       We have the following corollary which is a higher analogue of \cite[Lemma 10.20]{Bueh10}.   
      \begin{corollary}\label{cor:extensionclosed2}
         Suppose that $(\cD, \cX)$ is an $n$-exact category with \seg and $\cB \subset \cD$ is an $n$-extension closed additive subcategory.
         Then $(\cB, \cX_\cB)$ is an $n$-exact category with \seg, where $\cX_\cB \coloneqq \cX \cap \C^{n+2}_{\cB}$.
      \end{corollary}
      \begin{proof}
         By \Cref{lem:extclosedcoincide1} we know that $\cB$ is $n$-extension closed in $(\cD, \GG_\cX, \fr_\cX)$.
         \Cref{thm:extensionclosed} $(\cD, \GG_\cX, \fr_\cX)$ induces an $n$-exangulated structure $(\cB, \FF_\cB, \ft_\cB)$ on $\cB$.
         By \Cref{thm:correspondence}, any $\fr_\cX$-inflation is monic in $\cD$.
         The $\ft_\cB$-conflations are precisely the $\fr_\cX$-conflations with terms in $\cB$, see \Cref{rem:theremark}.
         Therefore, any $\ft_\cB$-inflation is monic in $\cD$ and hence in $\cB \subset \cD$.
         Dually, any $\ft_\cB$-deflation is epic in $\cB$.
         Hence, $(\cB, \FF_\cB, \ft_\cB)$ is $n$-exact, by \Cref{cor:nexact}.

         It follows from \Cref{thm:correspondence} that $(\cB, \FF_\cB, \ft_\cB)$ induces an $n$-exact structure $\cX_{(\FF_\cB, \ft_\cB)}$ with \seg on $\cB$.
         \Cref{rem:theremark} and $\cX = \cX_{(\GG_\cX, \fr_\cX)}$ imply $\cX_{\cB} = \cX_{(\FF_\cB, \ft_\cB)} $.
      \end{proof}


      \section{\texorpdfstring{$n$}{n}-extension closed subcategories of \texorpdfstring{$(n+2)$}{(n+2)}-angulated categories}\label{sec:nang}
       Throughout this section let $\cD$ be an additive category, $\Sigma \colon \cD \to \cD$ be an additive automorphism of $\cD$ and $\GG_{\Sigma}(-, -) \coloneqq \cD(-, \Sigma-)$ be the induced biadditive bifunctor, see \cite[Section 4.2]{HLN21}.
      We recall the following constructions from \cite[Section 4.2]{HLN21}. 

      Suppose $(\cD, \Sigma, \text{\pentagon})$ is an $(n+2)$-angulated category in the sense of \cite{GKO13}.
      Define a realisation $\fr_{\text{\pentagon}}$ of $\GG_\Sigma$ as follows.
      For $C,A \in \cD$ and $\delta \in \GG_\Sigma(C, A)$ pick an $(n+2)$-angle
      \begin{align}\widehat{X}_\bullet \colon& \hspace{1cm}\label{eqn:nangle}\begin{tikzcd}[ampersand replacement=\&]
        A \& {X_1} \& \cdots \& {X_n} \& C \& {\Sigma A}
        \arrow["\delta", from=1-5, to=1-6]
        \arrow["{d^X_0}", from=1-1, to=1-2]
        \arrow["{d^X_1}", from=1-2, to=1-3]
        \arrow["{d^X_{n-1}}", from=1-3, to=1-4]
        \arrow["{d^X_n}", from=1-4, to=1-5]
      \end{tikzcd}
      \intertext{in $\text{\pentagon}$. Let $X_\bullet \in \C^{n+2}_{\cD}$ be the truncated complex}
      X_\bullet \colon& \hspace{1cm}\label{eqn:shortnangle}\begin{tikzcd}[ampersand replacement=\&]
        A \& {X_1} \& \cdots \& {X_n} \& C
        \arrow["{d^X_0}", from=1-1, to=1-2]
        \arrow["{d^X_1}", from=1-2, to=1-3]
        \arrow["{d^X_{n-1}}", from=1-3, to=1-4]
        \arrow["{d^X_n}", from=1-4, to=1-5]
      \end{tikzcd}\end{align}
      and define $\fr_{\text{\pentagon}}(\delta) \coloneqq [X_\bullet]_{\cD}$.
      This is independent of the $(n+2)$-angle chosen in (\ref{eqn:nangle}), by \cite[Lemma 4.4]{HLN21}.
      Then $(\cD, \GG_\Sigma, \fr_{\text{\pentagon}})$ is $n$-exangulated, see \cite[Proposition 4.5]{HLN21}

      Conversely, let $(\cD, \GG_\Sigma, \fr)$ be $n$-exangulated.
      Let $\text{\pentagon}_{\fr}$ be the class of all complexes $\widehat{X}_\bullet$ as in (\ref{eqn:nangle}) such that $\langle X_\bullet, \delta \rangle$ is $\fr$-distinguished, where $X_\bullet$ is the corresponding complex in (\ref{eqn:shortnangle}), then $(\cD, \Sigma, \text{\pentagon}_{\fr})$ is $(n+2)$-angulated, see \cite[Proposition 4.8]{HLN21}.

      Indeed this gives us a bijective correspondence.
      \begin{theorem}[{\cite[Section 4.2]{HLN21}}]\label{prop:correspondence2}
         There is a one-to-one correspondence
         \begin{align*}
            \left\{\parbox{16.2em}{\centering $(n+2)$-angulated structures $(\cD, \Sigma, \text{\emph{\pentagon}})$} \right\} &\xleftrightarrow{\text{\emph{1:1}}} \left\{\parbox{15em}{\centering $n$-exangulated structures $(\cD, \GG_{\Sigma}, \fr)$}\right\} \\
                                                                                 (\cD, \Sigma, \text{\emph{\pentagon}}) &\xmapsto{\phantom{\text{\emph{1:1}}}} (\cD, \GG_\Sigma, \fr_\text{\emph{\pentagon}}) \\
                                                                    (\cD, \Sigma, \text{\emph{\pentagon}}_{\fr}) &\xmapsfrom{\phantom{\text{\emph{1:1}}}} (\cD, \GG_{\Sigma}, \fr).
        \end{align*}
      
      \end{theorem}
      \begin{proof}
         By \cite[Proposition 4.5]{HLN21}, every $(n+2)$-angulated structure $(\cD, \Sigma, \text{\pentagon})$ yields an $n$-exangulated structure $(\cD, \GG_\Sigma, \fr_{\text{\pentagon}})$.
         Conversely, by \cite[Propsition 4.8]{HLN21}, every $n$-exangulated structure $(\cD, \GG_\Sigma, \fr)$ yields an $(n+2)$-angulated structure $(\cD, \Sigma, \text{\pentagon}_{\fr})$.
         We only need to show $\text{\pentagon} = \text{\pentagon}_{\fr_{\text{\pentagon}}}$ for any $(n+2)$-angulated structure $(\cD, \Sigma, \text{\pentagon}_{\fr})$ and $\fr = \fr_{\text{\pentagon}_{\fr}}$ for any $n$-exangulated structure $(\cD, \GG_\Sigma, \fr)$.

         Let $(\cD, \GG_\Sigma, \fr)$ be $n$-exangulated, $A,C \in \cD$, $\delta \in \GG_{\Sigma}(C, A)$ and $\langle X_\bullet, \delta \rangle$ be $\fr$-distinguished.
         Then $X_\bullet$ is of the shape of (\ref{eqn:shortnangle}) and hence $\widehat{X}_\bullet$ as in (\ref{eqn:nangle}) is in $\text{\pentagon}_{\fr}$.
         Using the independence \cite[Lemma 4.4]{HLN21} provides, we have that $\langle X_\bullet, \delta \rangle$ is $\fr_{\text{\pentagon}_{\fr}}$-distinguished.
         Therefore, $\fr = \fr_{\text{\pentagon}_{\fr}}$.

         For the rest of this proof denote for any complex $\widehat{X}_\bullet$ as in (\ref{eqn:nangle}) the corresponding complex as in (\ref{eqn:shortnangle}) by $X_\bullet$.

         Let $(\cD, \Sigma, \text{\pentagon}_{\fr})$ be $(n+2)$-angulated.
         We show the two inclusions of $\text{\pentagon} = \text{\pentagon}_{\fr_{\text{\pentagon}}}$ separately. 
         Let $\widehat{X}_\bullet \in \text{\pentagon}$ be as in (\ref{eqn:nangle}).
         Then $[X_\bullet]_{\cD} = \fr_{\text{\pentagon}}(\delta)$, using \cite[Lemma 4.4]{HLN21}.
         Hence, $\widehat{X}_\bullet \in \text{\pentagon}_{\fr_{\text{\pentagon}}}$.
         Conversely, let $\widehat{X}_\bullet \in \text{\pentagon}_{\fr_{\text{\pentagon}}}$ be as in (\ref{eqn:nangle}).
         Then $\langle X_\bullet, \delta \rangle$ is $\fr_{\text{\pentagon}}$-distinguished, by definition.
         This means that there is a $\widehat{Y}_\bullet \in \text{\pentagon}$ of the shape
         \begin{align*}\widehat{Y}_\bullet \colon \hspace{1cm}&\begin{tikzcd}[ampersand replacement=\&]
           A \& {Y_1} \& \cdots \& {Y_n} \& C \& {\Sigma A}
            \arrow["\delta", from=1-5, to=1-6]
            \arrow["{d^Y_0}", from=1-1, to=1-2]
            \arrow["{d^Y_1}", from=1-2, to=1-3]
            \arrow["{d^Y_{n-1}}", from=1-3, to=1-4]
            \arrow["{d^Y_n}", from=1-4, to=1-5]
         \end{tikzcd}         
         \intertext{such that $[Y_\bullet]_{\cD} = [X_\bullet]_{\cD}$. Hence, there is a commutative diagram}
         &\begin{tikzcd}[ampersand replacement=\&]
           A \& {Y_1} \& \cdots \& {Y_n} \& C \& {\Sigma A}\\
           A \& {X_1} \& \cdots \& {X_n} \& C \& {\Sigma A}
            \arrow["\delta", from=1-5, to=1-6]
            \arrow["{d^Y_0}", from=1-1, to=1-2]
            \arrow["{d^Y_1}", from=1-2, to=1-3]
            \arrow["{d^Y_{n-1}}", from=1-3, to=1-4]
            \arrow["{d^Y_n}", from=1-4, to=1-5]
            \arrow["\delta", from=2-5, to=2-6]
            \arrow["{d^X_0}", from=2-1, to=2-2]
            \arrow["{d^X_1}", from=2-2, to=2-3]
            \arrow["{d^X_{n-1}}", from=2-3, to=2-4]
            \arrow["{d^X_n}", from=2-4, to=2-5]
            \arrow[from=1-1, to=2-1, equal]
            \arrow[from=1-2, to=2-2, dotted]
            \arrow[from=1-4, to=2-4, dotted]
            \arrow[from=1-5, to=2-5, equal]
            \arrow[from=1-6, to=2-6, equal]
         \end{tikzcd}         
         \end{align*}
         where the dotted morphisms are obtained through the homotopy equivalence $[X_\bullet]_{\cD} = [Y_\bullet]_{\cD}$.
         By \cite[Lemma 2.4]{GKO13} we have $\widehat{X}_\bullet \in \text{\pentagon}$.
         This shows $\text{\pentagon} = \text{\pentagon}_{\fr_{\text{\pentagon}}}$
      \end{proof}
      
       We recall the following definition.
      \begin{definition}\label{def:nangulatedextclosed}
         An additive subcategory $\cB \subset \cD$ of an $(n+2)$-angulated category $(\cD, \Sigma, \text{\pentagon})$ is called \emph{$n$-extension closed} if for all $A,C \in \cB$ and all $\delta \in \cD(C, \Sigma A)$ there is an $(n+2)$-angle $\widehat{X}_\bullet$ as in (\ref{eqn:nangle}) with $X_1, \dots, X_n \in \cB$. 
      \end{definition}

      \begin{remark}\label{rem:extclosed=extclosed}
         Suppose $(\cD, \Sigma, \pentagon)$ is an $(n+2)$-angulated category with an additive subcategory $\cB \subset \cD$.
         Then $\cB$ is $n$-extension closed in $(\cD, \Sigma, \pentagon)$ in the sense of \Cref{def:nangulatedextclosed} if and only if it is $n$-extension closed in $(\cD, \GG_\Sigma, \fr_{\pentagon})$ in the sense of \Cref{def:exangulatedextclosed}.
      \end{remark}

      Suppose $(\cD, \Sigma, \text{\pentagon})$ is an $(n+2)$-angulated category and $\cB \subset \cD$ is $n$-extension closed.
      For each $A,C \in \cB$ and $\delta \in \cD(C, \Sigma A)$ pick an $(n+2)$-angle $\widehat{X}_\bullet$ as in (\ref{eqn:nangle}) with $X_1, \dots, X_n \in \cB$ and define $\fr_{\cB}(\delta) = [X_\bullet]_{\cB}$, where $X_\bullet$ is the corresponding complex from (\ref{eqn:shortnangle}).
      This is well-defined using that for $A,C \in \cB$ and $\smash{X_\bullet, Y_\bullet \in \cC^{n+2}_{(\cB;A,C)}}$ the equality $[X_\bullet]_\cD = [Y_\bullet]_\cD$ implies $[X_\bullet]_\cB = [Y_\bullet]_\cB$ and using that $[X_\bullet]_\cD$ is independent of the choice of $\widehat{X}_\bullet \in \text{\pentagon}$ completing $\delta \colon C \to \Sigma A$ by \cite[Lemma 4.4]{HLN21}.
      The following corollary proves \cite[Theorem 1.2]{Zho22} in a more general setting.

      \begin{corollary}\label{cor:thefirstcorollary} 
         Suppose that $(\cD, \Sigma, \emph{\pentagon})$ is an $(n+2)$-angulated category and $\cB \subset \cD$ is an $n$-extension closed additive subcategory.
         Then $(\cB, \GG_{\cB}, \fr_{\cB})$ is an $n$-exangulated category, where $\GG_{\cB} = \GG_\Sigma|_{\scriptscriptstyle \cB^{\op} \times \cB}$ and $\fr_{\cB}$ is as defined above.
      \end{corollary}
      \begin{proof}
         By \Cref{prop:correspondence2} there is an $n$-exangulated structure $(\cD, \GG_{\Sigma}, \fr)$ on $\cD$ with $\text{\pentagon} = \text{\pentagon}_\fr$.
         By \Cref{rem:extclosed=extclosed} we know that $\cB \subset \cD$ is $n$-extension closed in $(\cD, \GG_{\Sigma},\fr)$.
         By \Cref{thm:extensionclosed} there is an $n$-exangulated structure $(\cB, \FF_{\cB}, \ft_{\cB})$ on $\cB$, where $\FF_{\cB} = \GG_\Sigma|_{\scriptscriptstyle \cB^{\op} \times \cB}$.
         It is clear that $\ft_{\cB}$ and $\fr_{\cB}$ coincide.
      \end{proof}

      Suppose $(\cD, \Sigma, \text{\pentagon})$ is an $(n+2)$-angulated category and $\cB \subset \cD$ is an $n$-extension closed additive subcategory with $\cD(\Sigma \cB, \cB) = 0$.
      Let $\cX_\cB$ be the class of all sequences $X_\bullet$ as in (\ref{eqn:shortnangle}) with $A, X_1, \dots, X_n, C \in \cB$ such that there exists a corresponding $(n+2)$-angle $\widehat{X}_\bullet$ as in (\ref{eqn:nangle}).
      The following corollary proves \cite[Theorem I]{Kla21} in a more general setting.

      \begin{corollary}  \label{recoverkla}
         Suppose that $(\cD, \Sigma, \emph{\pentagon})$ is an $(n+2)$-angulated category and $\cB \subset \cD$ is an $n$-extension closed additive subcategory with $\cD(\Sigma \cB, \cB) = 0$.
         Then $(\cB, \cX_\cB)$ is an $n$-exact category with \seg, where $\cX_{\cB}$ is as defined above.
      \end{corollary}
      \begin{proof}
         Let $(\cD, \GG_{\Sigma}, \fr_{\text{\pentagon}})$ be the $n$-exangulated structure induced on $\cD$ via \Cref{prop:correspondence2} and $(\cB, \GG_\cB, \fr_\cB)$ be the $n$-exangulated structure induced on $\cB$ via \Cref{cor:thefirstcorollary} or equivalently \Cref{thm:extensionclosed}.

         The class of $\fr_{\cB}$-conflations is the class of $\fr_{\text{\pentagon}}$-conflations with terms in $\cB$, by \Cref{rem:theremark}.
         By \Cref{prop:correspondence2}, the class of $\fr_{\text{\pentagon}}$-conflations is the class of all sequences $X_\bullet$ as in (\ref{eqn:shortnangle}) such that there exists a corresponding $(n+2)$-angle $\widehat{X}_\bullet$ as in (\ref{eqn:nangle}), which is in $\text{\pentagon}$.
         We conclude that the class of $\fr_{\cB}$-conflations is $\cX_{\cB}$.

         We show that all $\fr_\cB$-inflations are monic in $\cB$.
         Indeed let $A, X_1 \in \cB$ and $d^X_0 \colon A \to X_1$ be an $\fr_\cB$-inflation.
         Then there is an $(n+2)$-angle as in (\ref{eqn:nangle}) with $X_2, \dots, X_n, C \in \cB$.
         Applying the functor $\cD(-, B)$ for $B \in \cB$ yields an exact sequence
         \[\begin{tikzcd}[ampersand replacement=\&, column sep = large]
            {\cD(B, \Sigma^{-1}C)} \& {\cD(B,A)} \& {\cD(B, X_1)} 
            \arrow["{\cD(B, \Sigma^{-1}\delta)}", from=1-1, to=1-2]
            \arrow["{\cD(B, d^X_0)}", from=1-2, to=1-3]
         \end{tikzcd}\]         
         by \cite[Propositon 2.5]{GKO13}.
         Using that $\cB \subset \cD$ is full and $\cD(B, \Sigma^{-1}C) \cong \cD(\Sigma B, C) = 0$ because $\cD(\Sigma \cB, \cB) = 0$, we conclude that $d^X_0$ is monic in $\cB$.
         Similarly, one can show that $\fr_{\cB}$-deflations are epic.

         By \Cref{thm:correspondence} we conclude that $(\cB, \cX_\cB)$ is $n$-exact with \seg.
     \end{proof}

      \appendix \section{Calculations}

      \begin{calculation}\label{calc:biproduct}
         \Cref{dgr:biproduct} is a biproduct diagram in $\text{\AE}^{n+2}_{(\cC, \EE)}$.
      \end{calculation}
      \begin{proof}
         We first prove identities, which will be used later in the proof. We have
         \begin{align*}
            \id_{A} &= p' g, & e &= g p', & e' &= \id_{X_2} - gp' = \id_{X_2} - e, & p'h &= 0
         \intertext{by definition. Notice that $e$ and $e'$ are idempotents. The above identities imply $eg = g$ and $p' e = p'$ as well as $eh = 0$ which imply}
            e'g &= 0, & p'e' &= 0, & e'h &= h. &&
         \intertext{Finally, $d^{X''}_2 d^{X''}_1 = 0$ since $X''$ is a complex and hence}
               d^X_2 g &= 0, & d^X_2 h &=0, & d^X_2 e' &= d^X_2 (\id_{X_2} - gp') = d^X_2. &&
         \end{align*}

         Now, it is clear that all columns of \Cref{dgr:biproduct} except the third one are biproduct diagrams in $\cA$.
         Concerning the third column, we have
         \[\text{$\left[\begin{smallmatrix}p' & 0\end{smallmatrix}\right]\left[\begin{smallmatrix} g \\ 0\end{smallmatrix}\right] = p' g = \id_A$ and $\left[\begin{smallmatrix}e' & g\end{smallmatrix}\right]\left[\begin{smallmatrix} e' \\ p'\end{smallmatrix}\right]= e' + gp' = e' + e = \id_{X_2}$}\]
         as well as
         \[\left[\begin{smallmatrix}e' \\ p'\end{smallmatrix}\right]\left[\begin{smallmatrix}e' & g \end{smallmatrix}\right] +\left[\begin{smallmatrix} g \\ 0 \end{smallmatrix}\right]\left[\begin{smallmatrix} p' & 0 \end{smallmatrix}\right] = \left[\begin{smallmatrix} (e')^2 & e'g \\ p' e' & p' g \end{smallmatrix}\right] + \left[\begin{smallmatrix} gp' & 0 \\ 0 & 0 \end{smallmatrix}\right] = \left[\begin{smallmatrix} e' + e & 0 \\ 0 & \id_{A} \end{smallmatrix}\right] = \id_{X_2 \oplus A}.\]
         Hence, all columns of \Cref{dgr:biproduct} are biproduct diagrams in $\cA$.
         
         To conclude that \Cref{dgr:biproduct} is a biproduct in $\text{\AE}^{n+2}_{(\cC,\EE)}$ we only need to show that that all squares commute.
         The two upper left squares commute since
\[\text{$
\left[\begin{smallmatrix} 0 \\ \id_{\smash{X_1'}} \end{smallmatrix}\right]
f
=
\left[\begin{smallmatrix} 0 \\ f \end{smallmatrix}\right]
= d^{X''}_0
$ and $
\left[\begin{smallmatrix} 0 & \id_{\smash{X_1'}} \end{smallmatrix}\right]
d^{X''}_0
=
\left[\begin{smallmatrix} 0 & \id_{\smash{X_1'}} \end{smallmatrix}\right]
\left[\begin{smallmatrix} 0 \\ f \end{smallmatrix}\right]
=
f
$.}\]
The two lower left squares commute since
\[\text{$
\left[\begin{smallmatrix} \id_A & 0 \end{smallmatrix}\right]
d^{X''}_0 =
\left[\begin{smallmatrix} \id_A & 0 \end{smallmatrix}\right]
\left[\begin{smallmatrix} 0 \\ f \end{smallmatrix}\right]
=
0
$}\]
and any morphism starting in the zero object is $0$.
The two, second to left, upper squares commute since
\[\text{$
\left[\begin{smallmatrix} e' \\ p'\end{smallmatrix}\right]
h
=
\left[\begin{smallmatrix} e'h \\ p'h\end{smallmatrix}\right]
=
\left[\begin{smallmatrix} h \\ 0 \end{smallmatrix}\right]
=
\left[\begin{smallmatrix} g & h \\ 0 & 0 \end{smallmatrix}\right]
\left[\begin{smallmatrix} 0 \\ \id_{\smash{X_1'}} \end{smallmatrix}\right]
=
d^{X''}_1
\left[\begin{smallmatrix} 0 \\ \id_{\smash{X_1'}} \end{smallmatrix}\right]
$}\] 
and 
\[\text{$
\left[\begin{smallmatrix} e' & g\end{smallmatrix}\right]
d^{X''}_1
=
\left[\begin{smallmatrix} e' & g \end{smallmatrix}\right]
\left[\begin{smallmatrix} g & h \\ 0 & 0 \end{smallmatrix}\right]
=
\left[\begin{smallmatrix} e'g & e'h\end{smallmatrix}\right]
=
\left[\begin{smallmatrix} 0 & h\end{smallmatrix}\right]
=
h\left[\begin{smallmatrix} 0 & \id_{\smash{X_1'}} \end{smallmatrix}\right]
$.}\]
The two, second to left, lower squares commute since
\[\text{$
\left[\begin{smallmatrix} p' & 0 \end{smallmatrix}\right]
d^{X''}_1
=
\left[\begin{smallmatrix} p' & 0 \end{smallmatrix}\right]
\left[\begin{smallmatrix} g & h \\ 0 & 0 \end{smallmatrix}\right]
=
\left[\begin{smallmatrix} p'g & p'h \end{smallmatrix}\right]
=
\left[\begin{smallmatrix} \id_{A} & 0 \end{smallmatrix}\right]
$ and $
\left[\begin{smallmatrix} g \\ 0 \end{smallmatrix}\right]
=
\left[\begin{smallmatrix} g & h \\ 0 & 0 \end{smallmatrix}\right]
\left[\begin{smallmatrix} \id_A \\ 0 \end{smallmatrix}\right]
=
d^{X''}_1
\left[\begin{smallmatrix} \id_A \\ 0 \end{smallmatrix}\right]
$.}\]
The two, third to left, upper squares commute, since
\[\text{$
\left[\begin{smallmatrix} d_2^X e' \\ p' \end{smallmatrix}\right]
=
\left[\begin{smallmatrix} d_2^X & 0 \\ 0 & \id_{A} \end{smallmatrix}\right]
\left[\begin{smallmatrix} e' \\ p'\end{smallmatrix}\right]
=
d^{X''}_2
\left[\begin{smallmatrix} e' \\ p'\end{smallmatrix}\right]
$ and $
d^{X''}_2
= 
\left[\begin{smallmatrix} d_2^X & 0 \\ 0 & \id_{A} \end{smallmatrix}\right]
=
\left[\begin{smallmatrix} d_2^X(e')^2 & d^X_2 e' g \\ p' e' & p'g \end{smallmatrix}\right]
=
\left[\begin{smallmatrix} d_2^X e' \\ p' \end{smallmatrix}\right]
\left[\begin{smallmatrix} e' & g \end{smallmatrix}\right]
$.}\]
Finally, the two, third to left, lower squares commute because any morphism ending in the zero object is $0$ and
\[0 = 
\left[\begin{smallmatrix} d_2^X g & 0 \\ 0 & 0 \end{smallmatrix}\right]
=
\left[\begin{smallmatrix} d_2^X & 0 \\ 0 & \id_{A} \end{smallmatrix}\right]
\left[\begin{smallmatrix} g \\ 0 \end{smallmatrix}\right]
=
d^{X''}_{2}
\left[\begin{smallmatrix} g \\ 0 \end{smallmatrix}\right].
\]
It is clear that the remaining squares commute.
It follows immediately that \Cref{dgr:biproduct} is a biproduct diagram in $\text{\AE}^{n+2}_{(\cC,\EE)}$.
      \end{proof}
      
      \begin{calculation}\label{calc:biproduct2}
         \Cref{dgr:biproduct2} is a biproduct diagram in $\text{\AE}^{3}_{(\cC, \EE)}$.
      \end{calculation}
      \begin{proof}
      By construction, all columns of \Cref{dgr:biproduct2} are biproduct diagrams in $\cA$, all squares except for two upper right squares commute and $(j')^\ast \delta = \delta'$.
      We only need to show that the two upper right squares commute, that the upper row of \Cref{dgr:biproduct2} is an $\EE$-attached complex, and that $(d_1^X i)^\ast \delta = 0$ and $(q')^\ast \delta' = \delta$ hold.
      As the columns of \Cref{dgr:biproduct2} are biproduct diagrams the identities
      \begin{align*}
        qj &= \id_{X_1'}, & ip + jq &= \id_{A \oplus X_1'}, & e &= d_1^X i p', & q'j' &= \id_{X_2'}, & e' &= \id_{X_2} - e = j' q'
      \end{align*}
      hold.
      We have $e d_1^X = (d_1^X i p') d_1^X = d_1^X (i p)$ as the two lower left squares commute.
      Hence,
      \[ j' (q'd_1^X j) = e' d_1^X j = (\id_{X_2} - e) d_1^X j =  d_1^X (\id_{A \oplus X_1'} - ip) j = d_1^X j q j = d_1^X j \]
      as well as
      \[ q' d_1^X = q'j'q' d_1^X = q' e' d_1^X = q' (\id_{X_2} - e) d_1^X = q' d_1^X (\id_{A \oplus X_1'} - ip) = (q' d^X_1 j) q \]
      show that the two upper right squares commute.
      We have $(d_1^X)^\ast \delta = 0 $ as $\langle X_\bullet, \delta \rangle$ is a $1$-exangle.
      Hence, $(q' d_1^X j)^\ast \delta' = (j' q' d_1^X j)^\ast \delta = (d_1^X j)^\ast \delta = 0$ and $(d_1^X i)^\ast \delta = 0$.
      In particular, the upper row of \Cref{dgr:biproduct2} is an $\EE$-attached complex.
      Finally, $(d_1^X)^\ast \delta = 0$ implies
      \[(q')^\ast \delta' = (j' q')^\ast \delta = (\id_{X_2} -e)^\ast \delta = \delta - (d^X_1 i p')^\ast \delta = \delta \]
      which completes the proof.
    \end{proof}

   \begin{acknowledgment}
      Many thanks to Peter Jørgensen for his mathematical and linguistic advice and corrections.
      Thanks to Johanne Haugland for inspiring me to write this paper.
      Also thanks to her for suggestions and email communication.
      Thanks to Amit Shah for pointing out the importance of condition (WIC) to me.

      This work was supported by Aarhus University Research Foundation, grant no. AUFF-F-2020-7-16.
   \end{acknowledgment}

      \bibliographystyle{alpha}
      
\end{document}